\documentclass[10pt,a4paper]{article}

\usepackage{amsmath,amssymb,amsthm,mathrsfs}
\usepackage{stmaryrd}  
\usepackage{graphicx}
\usepackage{xcolor}
\usepackage{fullpage}
\usepackage{bm}
\usepackage{cite}
\usepackage{tabularx}

\newlength{\defbaselineskip}
\newcommand{\setlinespacing}[1]%
{\setlength{\baselineskip}{#1 \defbaselineskip}}

\theoremstyle{plain}
\newtheorem{theorem}{Theorem}[section]
\newtheorem{lemma}[theorem]{Lemma}
\newtheorem{proposition}[theorem]{Proposition}
\newtheorem{corollary}[theorem]{Corollary}

\theoremstyle{definition}

\newtheorem{assumption}{Assumption}

\theoremstyle{remark}
\newtheorem{remark}[theorem]{Remark}

\numberwithin{equation}{section}

\DeclareMathOperator*{\esssup}{ess\,sup}

\allowdisplaybreaks

\begin{document}

\title{Mean Field Exponential Utility Game: A Probabilistic Approach\thanks{The first and the third authors gratefully acknowledge financial support from the Singapore MOE AcRF grant R-146-000-271-112.  The second author gratefully acknowledges the PhD scholarship from NUS and financial support from the Centre for Quantitative Finance, NUS. In addition, the third author also received support from the NSFC grant 11871364.}}
\author{Guanxing Fu\footnote{Department of Mathematics, National University of Singapore; email: fuguanxing725@gmail.com} \qquad Xizhi Su\footnote{Department of Mathematics, National University of Singapore; email: suxizhi@u.nus.edu} \qquad Chao Zhou\footnote{Department of Mathematics, Institute of Operations Research \& Analytics and Suzhou Research Institute, National University of Singapore; email: matzc@nus.edu.sg}}

\maketitle

\begin{abstract}
We study an $N$-player and a mean field exponential utility game. Each player manages two stocks; one is driven by an individual shock and the other is driven by a common shock. Moreover, each player is concerned not only with her own terminal wealth but also with the relative performance of her competitors. We use the probabilistic approach to study these two games. We show the unique equilibrium of the $N$-player game and the mean field game can be characterized by a novel multi-dimensional FBSDE with quadratic growth and a novel mean-field FBSDEs, respectively. The well-posedness result and the convergence result are established. 
\end{abstract}




\section{Introduction}
In this paper we consider the following $N$-player exponential utility game 
\begin{equation}\label{eq:intro-N-game}
	\left\{\begin{split}
			&~\mathbb E\left[-e^{-\alpha^i(X^i_T-\theta^i\overline X^{-i}_T)}\right]\rightarrow\max,\\
			&~\textrm{subject to }dX^i_t=\pi^i_t\,\frac{dS^i_t}{S^i_t}+\pi^{i0}_t\,\frac{dS^{i0}_t}{S^{i0}_t},\quad X^i_0=x^i,\\
			&~i=1,\cdots,N,
	\end{split}\right.
\end{equation}
as well as its mean field game (MFG) correspondence
\begin{equation}\label{eq:intro-MFG}
	\left\{\begin{split}
1.&~\textrm{Fix }\mu.\\
2.&~\mathbb E\left[-e^{-\alpha(X_T-\theta\mu)}\right]\rightarrow\max,\\
&~\textrm{subject to }dX_t=\pi_t\,\frac{dS_t}{S_t}+\pi^0_t\,\frac{dS^0_t}{S^0_t},\quad X_0=x.\\
3.&~\mu=Law(X^*_T) \textrm{ with }X^*\textrm{ the optimal state}.
\end{split}\right.
\end{equation}
In \eqref{eq:intro-N-game} $X^i$ is the wealth process of player $i$, $S^i$ and $S^{i0}$ are the price dynamics for the stocks managed by player $i$,  driven by an individual shock and a common shock, respectively, and $\pi^i$ and $\pi^{i0}$ are the amount of portfolios invested in these two stocks, respectively. For simplicity, we assume the interest rate is $0$. In \eqref{eq:intro-N-game}, each player is not only concerned with her own terminal wealth $X^i_T$, but also concerned with the relative performance of her competitors $(X^i-\overline X^{-i})$, with $\overline X^{-i}=\frac{1}{N-1}\sum_{j\neq i}X^j$ being the average wealth of player $i$'s competitors. The random variable $\theta^i$ valued in $[0,1]$ is the relative performance factor for player $i$; player $i$ is more concerned about the relative performance when $\theta^i$ is larger. Similar problems have been considered in Espinosa \& Touzi \cite{ET-2015}, Frei \cite{Frei2014}, Frei \& dos Reis \cite{FR-2011}, Gu\'eant, Lasry \& Lions \cite{GLL-2011} and Lacker \& Zariphopoulou \cite{LZ-2019}. In \cite{ET-2015,Frei2014,FR-2011}, it is assumed that all players manage common stocks, i.e., $\pi^i\equiv 0$ and $S^{i0}=S$ for all $i$,  with various trading constraints. In particular, using a BSDE approach \cite{ET-2015} shows the existence and uniqueness of a Nash equilibrium for unconstraint agents with general utility and constrained agents with exponential utility in Black-Scholes model (deterministic return rate and volatility). Moreover, a convergence result as the number of players goes to infinity is established as well. 
\cite{FR-2011} considers a similar problem as \cite{ET-2015} and shows when the equilibrium does not exist by constructing counterexamples, where the BSDE characterizing the equilibrium has no solution. In \cite{Frei2014}, a similar problem as in \cite{ET-2015,FR-2011} is considered to illustrate the notion of split solutions of BSDEs.
In \cite{GLL-2011}, a static competitative Markowitz model is studied. In \cite{LZ-2019},  each player is assumed to trade one  stock that is correlated with the stocks managed by the competitors. Under this framework an $N$-player game and an MFG are analyzed in a constant setting, i.e., all the coefficients are assumed to be independent of time. Recently, dos Reis \& Platonov in \cite{Reis2020} study a similar MFG under relative performance with forward utilities, also under constant setting. 

 In this paper, using a probabilistic approach we analyze \eqref{eq:intro-N-game} and \eqref{eq:intro-MFG} in a general non-Markovian setting, which is beyond the constant setting. Unlike \cite{ET-2015,FR-2011,Frei2014}, the trading constraint is not considered here. We characterize the unique equilibrium of \eqref{eq:intro-N-game} in terms of a novel multi-dimensional FBSDE with quadratic growth and possibly unbounded terminal condition. The unique equilibrium of the MFG correspondence \eqref{eq:intro-MFG} is characterized in terms of a novel mean field FBSDE. Note that we obtain the FBSDE characterization by using martingale optimality principle while \cite{ET-2015} and \cite{FR-2011} obtain the BSDE characterization by establishing a dynamic programming principle combined with martingale representation and martingale optimality principle. Table \ref{table:comparison} compares our model with some closest literatures. 
The contribution in this paper is three-fold. First, we find the unique Nash equilibrium of the $N$-player game \eqref{eq:intro-N-game} by establishing the well-posedness result of a novel multi-dimensional FBSDE with quadratic growth. Moreover, a convergence result is established to show the equilibrium of the $N$-player game \eqref{eq:intro-N-game} converges to the equilibrium of the MFG correspondence \eqref{eq:intro-MFG}. Second, under a mild assumption on the relative performance factor $\theta$, we solve the MFG \eqref{eq:intro-MFG}. Moreover, when the assets' dynamics of all players are independent, the MFG \eqref{eq:intro-MFG} can be solved with more general $\theta$, which is allowed to be measurable to the largest $\sigma$-algebra. In this case, we solve \eqref{eq:intro-MFG} under a {weak interaction} assumption, that is, when $\theta$ is small. In particular, we solve a mean field FBSDE and it has its own interest to establish the global well-posedness result. Finally, when the return rates are independent of the noise, we obtain the equilibria for both the $N$-player game and the MFG in closed forms, which are beyond constant settings as in \cite{LZ-2019} and deterministic settings as in \cite{ET-2015}. Our equilibrium formulas for the $N$-player game and the MFG share similar structures; the investment strategy for the stock driven by the individual shock is of Merton type and the investment strategy for the stock driven by the common shock is a weighted sum of a Merton portfolio and the aggregation of the competitors' strategies.

\begin{table}[h] \label{table:comparison}
	\centering
	\caption{Comparison with Utility Games Closest to This Paper}
	\begin{tabular}{p{4.4cm}|p{4cm}|p{3cm}|p{2.5cm}} 

		\hline
		\hline
		Literature & Number of Stock Types & Characterization & Non-Markovian Setting \\ 
		\hline 
		Espinosa \& Touzi \cite{ET-2015}& 1 common type & BSDE & Yes/No\\
		\hline
		Frei \& dos Reis \cite{FR-2011}& 1 common type & BSDE & Yes\\
		\hline
		Lacker \& Zariphopoulou \cite{LZ-2019} & 1 correlated stock & PDE & No\\
		\hline
	    This paper & 2 correlated stocks & FBSDE & Yes\\
		\hline
		\hline
	\end{tabular}
\end{table}

Throughout we work on the  filtered probability space $(\Omega,\mathbb P,\mathbb G)$, on which a sequence of independent Brownian motions $\{W,W^0,\{W^n\}_n\}$ are defined. We will use $W^0$ to characterize the common shock of the stock price for all players, use $W^i$ to characterize the private shock of the stock price for player $i$ in the $N$-player game and use $W$ to characterize the private shock of the stock price for the representative player in the MFG. Denote by $\mathbb F^{(N)}$ the filtration generated by $(W^1,\cdots,W^N,W^0)$, by $\mathbb F^i$ the filtration generated by $(W^i,W^0)$, by $\mathbb F$ the filtration generated by $(W,W^0)$ and by $\mathbb F^{0}$ the filtration generated by $W^0$. All filtrations are understood in a completion sense. In the next two subsections, we introduce the $N$-player game and the MFG in detail.

\subsection{The $N$-Player Utility Game}\label{subsec:N-game}
 Let the price dynamics for the two stocks managed by player $i$ satisfy
\begin{equation}\label{price-N-i}
\frac{dS^i_t}{S^i_t}=b^i_t\,dt+\sigma^i_t\,dW^i_t\quad\textrm{ and }\quad\frac{dS^{i0}_t}{S^{i0}_t}=b^{i0}_t\,dt+\sigma^{i0}_t\,dW^0_t
\end{equation}
where $(b^i,b^{i0})$ and $(\sigma^i,\sigma^{i0})$ are the random return rates and volatilities of the two stocks and  $\overline W^i:=(W^i,W^0)^\top$. Assume the volatilities $\sigma^i$ and $\sigma^{i0}$ are non-degenerate. Let $\pi^i$ and $\pi^{i0}$ be the amount of money invested into the two stocks. The wealth process of player $i$ follows
\begin{equation}\label{wealth-exp-N-i}
	\begin{split}
dX^i_t=\pi^i_t(b^i_t\,dt+\sigma^i_tdW^i_t)+\pi^{i0}_t(b^{i0}_t\,dt+\sigma^{i0}_tdW^0_t)= \overline \pi^i_t(\overline b^i_t\,dt+\,d\overline W^i_t),\quad X^i_0=x^i,
	\end{split}
\end{equation}
where 
\begin{equation*}
	\begin{split}
	\overline \pi^i=(\sigma^i\pi^i,\sigma^{i0}\pi^{i0})\quad\textrm{ and }\quad
	\overline b^i=(b^i/\sigma^i,b^{i0}/\sigma^{i0})^\top. 
	\end{split}
\end{equation*}
Without loss of generality, from now on we assume $\sigma^i=\sigma^{i0}=1$. 

 The goal for player $i$ is to maximize the exponential utility with relative concern by choosing $(\pi^i,\pi^{i0})$ or $\overline\pi^i$
\begin{equation}\label{N-cost-i}
	\mathbb E\left[-e^{-\alpha^i(X^i_T-\theta^i\overline X^{-i}_T)}\right]\rightarrow\max.
\end{equation}
The following heuristic argument shows that the unique Nash equilibrium of \eqref{wealth-exp-N-i}-\eqref{N-cost-i} is the vector   
\[
	(\overline \pi^{1,*},\cdots,\overline \pi^{i,*}),
\] 
where for each $i=1,\cdots,N$
\begin{equation}\label{candidate-equilibrium-N}
\overline \pi^{i,*}=\overline Z^i+\frac{(\overline b^i)^\top}{\alpha^i}
\end{equation}
and $(X^i,Y^i,\overline Z^i=(Z^{ii},Z^{i0}),(Z^{ij})_{j\neq i})_{i=1,\cdots,N}$ is the unique solution to the following multi-dimensional FBSDE system
\begin{equation}\label{CN:exp-FBSDE-N}
\left\{\begin{split}
X^i_t=&~x^i+\int_0^t\left(\overline Z^{i}_s\overline b^i_s+\frac{|\overline b^i_s|^2}{\alpha^i}\right)\,ds+\int_0^t\left(\overline Z^{i}_s+\frac{(\overline b^i_s)^\top}{\alpha^i}\right)\,d\overline W^i_s,\\
Y^i_t=&~\theta^i\overline X^{-i}_T-\int_t^T\left(\overline Z^{i}_s\overline b^i_s+\frac{|\overline b^i_s|^2}{2\alpha^i}-\frac{\alpha^i}{2}\sum_{j\neq i}(Z^{ij}_s)^2\right)\,ds-\int_t^T\overline Z^{i}_s\,d\overline W_s^i-\sum_{j\neq i}\int_t^TZ^{ij}_s\,dW^j_s,\\
i=&~1,\cdots,N.
\end{split}\right.
\end{equation}
Fixing the competitors' strategies $(\overline \pi^j)_{j\neq i}$ and letting $(X^j)_{j\neq i}$ be the corresponding wealth, following the martingale optimality principle in Hu, Imkeller \& M\"uller \cite{HIM-2005} (see also Mania \& Schweizer \cite{MS-2005} and Rouge \& El Karoui \cite{Rouge2000}) we construct a family of stochastic processes $R^{i,\overline \pi}$ such that 
\begin{itemize}
	\item $R^{i,\overline \pi}$ is a supermartingale for all $\overline \pi$ and $R^{i,\overline \pi^{i,*}}$ is a martingale for some $\overline \pi^{i,*}$,
	\item $R^{i,\overline \pi}_0=R_0$ for all $\overline \pi$  and
	\item $R^{i,\overline \pi}_T=-e^{-\alpha^i(X^i_T-\theta^i\overline X^{-i}_T)}$.
\end{itemize}
Construct $R^{i,\overline \pi}$ in terms of a stochastic process $Y^i$ via $R^{i,\overline \pi}=-e^{-\alpha^i(X^{i,\overline \pi}-Y^i)}$.
From the above three items, it holds that $\mathbb E[-e^{-\alpha^i(X^{i,\overline \pi}_T-\theta^i\overline X^{-i}_T)}]\leq -e^{-\alpha^i(x^i-Y^i_0)}=\mathbb E[-e^{-\alpha^i(X^{i,\overline \pi^{i,*}}_T-\theta^i\overline X^{-i}_T)}]$ for all $\overline \pi$. Thus, $\overline \pi^{i,*}$ is optimal for player $i$'s optimization problem. Assume the diffusion process $Y^i$ admit the following expression
\[
	dY^i_t=f^i_t\,dt+\overline Z^{i}_t\,d\overline W^i_t+\sum_{j\neq i}Z^{ij}_t\,dW^j_t.
\]
We will find $f^i$ such that $Y^i$ satisfies \eqref{CN:exp-FBSDE-N} and the constructed $R^{i,\overline \pi}$ satisfies the above three items. By the construction,
\begin{equation*}
	\begin{split}
	R_t^{i,\overline \pi}=&~-e^{-\alpha^i\left(x^i-Y^i_0+\int_0^t(\overline \pi_s\overline b^i_s-f^i_s)\,ds+\int_0^t(\overline \pi_s-\overline Z^i_s)\,d\overline W^i_s-\sum_{j\neq i}\int_0^tZ^{ij}_s\,dW^j_s\right)}\\
	=&~-e^{-\alpha^i(x^i-Y^i_0)}e^{-\alpha^i\int_0^t(\overline \pi_s-\overline Z^i_s)\,d\overline W^i_s-\frac{(\alpha^i)^2}{2}\int_0^t|\overline \pi_s-\overline Z^i_s|^2\,ds+\alpha^i\sum_{j\neq i}\int_0^tZ^{ij}_s\,dW^j_s-\frac{(\alpha^i)^2}{2}\sum_{j\neq i}\int_0^t(Z^{ij}_s)^2\,ds}\\
	&~\times e^{-\alpha^i\int_0^t(\overline \pi_s\overline b^i_s-f^i_s)\,ds+\frac{(\alpha^i)^2}{2}\int_0^t|\overline \pi_s-\overline Z^i_s|^2\,ds+\frac{(\alpha^i)^2}{2}\sum_{j\neq i}\int_0^t(Z^{ij}_s)^2\,ds}\\
	=&~M^{i,\overline \pi}_tA^{i,\overline \pi}_t,
	\end{split}
\end{equation*}
where
\[
	M_t^{i,\overline\pi}=-e^{-\alpha^i(x^i-Y^i_0)}e^{-\alpha^i\int_0^t(\overline \pi_s-\overline Z^i_s)\,d\overline W^i_s-\frac{(\alpha^i)^2}{2}\int_0^t|\overline \pi_s-\overline Z^i_s|^2\,ds+\alpha^i\sum_{j\neq i}\int_0^tZ^{ij}_s\,dW^j_s-\frac{(\alpha^i)^2}{2}\sum_{j\neq i}\int_0^t(Z^{ij}_s)^2\,ds}
\]
and
\[
	A_t^{i,\overline\pi}=e^{-\alpha^i\int_0^t(\overline \pi_s\overline b^i_s-f^i_s)\,ds+\frac{(\alpha^i)^2}{2}\int_0^t|\overline \pi_s-\overline Z^i_s|^2\,ds+\frac{(\alpha^i)^2}{2}\sum_{j\neq i}\int_0^t(Z^{ij}_s)^2\,ds}.
\]
Take it for grant that the local martingale $M^{i,\overline \pi}$ is a true martingale. To make $R^{i,\overline \pi}$ a supermartingale for all $\overline\pi$ and a maringale for some $\overline\pi^{i,*}$, it is sufficient to let $A^{i,\overline \pi}$ be non-decreasing for all $\overline \pi$ and $1$ for some $\overline \pi^{i,*}$. This is equivalent to letting the integrand of $\log A^{i,\overline\pi}$ be non-negative for all $\overline \pi$ and $0$ for some $\overline \pi^{i,*}$.  Indeed,
\begin{equation*}
	\begin{split}
	&~-\alpha^i(\overline \pi\overline b^i-f^i)+\frac{(\alpha^i)^2}{2}|\overline \pi-\overline Z^i|^2+\frac{(\alpha^i)^2}{2}\sum_{j\neq i}(Z^{ij})^2\\
	=&~\frac{(\alpha^i)^2}{2}\left|\overline \pi-\overline Z^i-\frac{(\overline b^i)^\top}{\alpha^i}\right|^2-\frac{|\overline b^i|^2}{2}-\alpha^i\overline Z^i\overline b^i+\alpha^if^i+\frac{(\alpha^i)^2}{2}\sum_{j\neq i}(Z^{ij})^2,
	\end{split}
\end{equation*}
which implies \eqref{candidate-equilibrium-N} and
\[
	 f^i=\frac{|\overline b^i|^2}{2\alpha^i}+\overline Z^i\overline b^i-\frac{\alpha^i}{2}\sum_{j\neq i}(Z^{ij})^2.
\]
Thus, the candidate Nash equilibrium $(\overline\pi^{1,*},\cdots,\overline \pi^{N,*})$ and the value function $-e^{-\alpha^i(x^i-Y^i_0)}$ for player $i$ are characterized by the multi-dimensional (coupled) FBSDE with quadratic growth and unbounded terminal condition \eqref{CN:exp-FBSDE-N}, which is in a non-Markovian setting.

The literature on multi-dimensional FBSDEs with quadratic growth is sparse. Among the rare results, Fromm \& Imkeller \cite{Fromm2013} and Kupper, Luo \& Tangpi \cite{Kupper2019} considered Markovian cases, Fromm \& Imkeller \cite{Fromm2017} studied a ``non-Markovian'' case arising in utility maximization, where the randomness comes from an exogenous diffusion, Herdegen, Muhle-Karbe \& Possama\"i \cite{Herdegen2019} studied a two-dimensional non-Markovian case arising in equilibrium pricing with transaction cost, and Luo \& Tangpi \cite{Luo2017} considered a non-Markovian case but with a bounded terminal condition. 

Because of the unboundedness of the terminal condition $\theta^i\overline X^{-i}_T$ and the non-Markovian nature of \eqref{CN:exp-FBSDE-N}, it is difficult to solve \eqref{CN:exp-FBSDE-N} directly by the methods in the above literature. Based on the specific structure of our system, using a transformation, we show \eqref{CN:exp-FBSDE-N} is equivalent to a non-Markovian multi-dimensional BSDE with quadratic growth and bounded terminal condition. Moreover, in addition to the well-posedness result we also expect \eqref{CN:exp-FBSDE-N} to converge to the mean field system characterizing the equilibrium of the MFG correspondence. 
Motivated by this expectation, we solve the resulting  multi-dimensional BSDE by considering a sequence of benchmark mean field BSDEs firstly and then by applying BMO and fixed point analysis, which is inspired by Tevzadze \cite{Tevzadze2008}, to the difference of the multi-dimensional BSDE and the benchmark ones; see Section \ref{sec:exp-N} for details. 


Before turning to the next section to introduce the MFG in detail, we should mention that
in addition to the aforementioned \cite{ET-2015,Frei2014,FR-2011},  non-Markovian multi-dimensional quadratic BSDEs have also been studied by Bielagk, Lionnet \& dos Reis \cite{BLDR-2017}, Elie \& Possama\"i \cite{Elie2019}, Harter \& Richou \cite{HR-2019}, Hu \& Tang \cite{HT-2016}, Jamneshan, Kupper \& Luo \cite{Jamneshan2017}, Kardaras, Xing \& \v Zitkovic \cite{Kardaras2015}, Kramkov \& Pulido \cite{Kramkov2016}, Nam \cite{Nam-2019} among others. But none of these results covers ours, especially in veiw of that our multi-dimensional system converges to a mean field system.

\subsection{The Mean Field Utility Game}
Our mean field utility game is in the framework of MFGs as introduced in Huang, Malhame \& Caines \cite{HMC-2006} and Lasry \& Lions \cite{LL-2007} (see also Bensoussan, Frehse \& Yam \cite{BFY-2013}, Cardaliaguet \cite{Cardaliaguet-2013} and Carmona \& Delarue \cite{CD-2018a,CD-2018b}). The main idea is to decouple local dynamics by global dynamics and consider a representative player's optimization problem by solving only one equation instead of a couple multi-dimensional system. The equilibrium is established by solving a fixed point problem. The probabilistic method we will apply is inspired by Carmona \& Delarue  \cite{CD-2013} and \cite{CD-2018a}, although we incoporate FBSDEs through the martingale optimality principle rather than the maximum principle.

In the MFG correspondence, the price dynamics for the two stocks managed by the representative player follow
\begin{equation}\label{price-MFG}
\frac{dS_t}{S_t}=b_t\,dt+\,dW_t\quad\textrm{ and }\quad\frac{dS^0_t}{S^0_t}=b^0_t\,dt+\,dW^0_t,
\end{equation}
where $b$ and $b^0$ are randon return rates and $\overline W:=(W,W^0)$.
Let $\pi$ and $\pi^0$ be the amount of money invested into the two stocks, respectively. 
The dynamics of the representative player's wealth follows
\begin{equation}\label{wealth-exponential-common-noise}
dX_t=\pi_t(b_t\,dt+\,dW_t)+\pi^0_t(b^0_t\,dt+\,dW^0_t)=\overline \pi_t(\overline b_t\,dt+\,d\overline W_t),
\end{equation}
where
\begin{equation*}
\begin{split}
\overline \pi=(\pi,\pi^{0})\quad\textrm{ and }\quad
\overline b=(b,b^{0})^\top.
\end{split}
\end{equation*}
Again, the volatilities in \eqref{price-MFG} are assumed to be $1$ for simplicity. 
We consider the mean field utility game: 
\begin{equation}\label{exp-MFG}
\left\{\begin{split}
1.&~\textrm{Fix a }\mathcal F^0_T\textrm{ adapted }\mu;\\
2.&~\textrm{Solve the optimization problem: }\max_{\overline \pi}\mathbb E\left[-e^{-\alpha(X_T-\theta\mu)}\right]\textrm{ subject to }\eqref{wealth-exponential-common-noise};\\
3.&~\textrm{Search for the fixed point }\mu=\mathbb E[X_T^*|\mathcal F^0_T],\textrm{ where }X^*\textrm{ is the optimal wealth from 2.}
\end{split}\right.
\end{equation}
The probabilistic approach of MFG and the argument in Hu, Imkeller \& M\"uller \cite[Section 2]{HIM-2005} yield the following mean field FBSDE
\begin{equation}\label{FBSDE-exponential-common-noise}
\left\{\begin{split}
X_t=&~\mathcal X+\int_0^t\left(\overline Z_s+\frac{\overline b^\top_s}{\alpha}\right)(\overline b_s\,ds+\,d\overline W_s),\\
Y_t=&~\theta\mathbb E[X_T|\mathcal F^0_T]-\int_t^T\left(\overline Z_s\overline b_s+\frac{|\overline b_s|^2}{2\alpha}\right)\,ds-\int_t^T\overline Z_s\,d\overline W_s,
\end{split}\right.
\end{equation}
where $\overline Z=(Z,Z^0)$ and the candidate optimal strategy is given by
\[
	\overline \pi^*=\overline Z+\frac{\overline b^\top}{\alpha}.
\]
In Section \ref{sec:MFG-exp-CN} we will solve \eqref{FBSDE-exponential-common-noise} by using a similar transformation as that to solve \eqref{CN:exp-FBSDE-N}, but under much weaker assumptions than those under which \eqref{CN:exp-FBSDE-N} is solvable. This provides sufficient reason to consider the MFG \eqref{exp-MFG} even if the corresponding $N$-player game \eqref{wealth-exp-N-i}-\eqref{N-cost-i} is solvable and the convergence result is established.

 Moreover, when the dynamics of the stocks managed by each player is independent, that is each player is assumed to manage only one stock driven by the individual shock $W$ or $\pi^0\equiv 0$, the competition factor can be assumed to be measurable w.r.t. the largest $\sigma$-algebra in the resulting MFG
\begin{equation}\label{exp-MFG-IA}
\left\{\begin{split}
1.&~\textrm{Fix a }\mu\in\mathbb R;\\
2.&~\textrm{Solve the optimization problem: }\max_{\pi}\mathbb E\left[-e^{-\alpha(X_T-\theta\mu)}\right]\\
&~\textrm{subject to }\eqref{wealth-exponential-common-noise}\textrm{ with }\pi^0=0;\\
3.&~\textrm{Search for the fixed point }\mu=\mathbb E[X_T^*],\textrm{ where }X^*\textrm{ is the optimal wealth from 2.}
\end{split}\right.
\end{equation}
The resulting mean field FBSDE follows
\begin{equation}\label{FBSDE-exp-IA}
\left\{\begin{split}
X_t=&~\mathcal X+\int_0^t\left( Z_s+\frac{ b_s}{\alpha}\right)(b_s\,ds+\,dW_s),\\
Y_t=&~\theta\mathbb E[X_T]-\int_t^T\left( Z_s b_s+\frac{| b_s|^2}{2\alpha}\right)\,ds-\int_t^T Z_s\,dW_s,
\end{split}\right.
\end{equation}
which cannot be transformed into a mean field BSDE like \eqref{CN:exp-FBSDE-N} and \eqref{FBSDE-exponential-common-noise} because $\theta$ is measurable w.r.t. the largest $\sigma$-algebra. Although in a linear form, the mean field FBSDE \eqref{FBSDE-exp-IA} does not appear in the literature due to its random and non-monotone coefficients as well as the appearance of $Z$ in the drift. In particular, the continuation method to obtain global solutions in \cite{ARY-2019,BYZ-2015,CD-2015,FGHP-2018,FH-2018}  does not work here. 
The solvability of \eqref{FBSDE-exp-IA} on the whole interval $[0,T]$ has its own interest. 
Our approach is motivated by Ma, Wu, Zhang \& Zhang \cite{MWZZ-2015}, where there is no mean field term. First, we solve \eqref{FBSDE-exp-IA} on a short time interval by a contraction argument. Second, we extend the local solution to the global one by establishing the existence result of a decoupling field with Lipschitz property. In order to do so, we need to analyze the associated characteristic (mean field) BSDE; see Section \ref{sec:MFG-IA-exp} for details.

The rest of this paper is organized as follows. In Section \ref{sec:exp-N}, we solve the multi-dimensional FBSDE with quadratic growth \eqref{CN:exp-FBSDE-N} and the $N$-player game \eqref{wealth-exp-N-i}-\eqref{N-cost-i}. The convergence result is also established there. Under mild assumptions, in Section \ref{sec:MFG-exp-CN} and Section \ref{sec:MFG-IA-exp} we solve the MFG \eqref{exp-MFG} and the MFG \eqref{exp-MFG-IA}, respectively. In Section \ref{sec:example}, explicitly solvable examples and the financial interpretation are provided.
\paragraph{Notation of Spaces.}  Let $\mathbb K$ be a generic filtration and $\mathbb I\subset [0,T]$ be an interval. Define 
\[
	\mathbb L_{\mathbb K}(\mathbb I)=\left\{ P:\Omega\times\mathbb I\rightarrow\mathbb R: P\textrm{ is progressively measurable w.r.t. }\mathbb K \right\},
\]
\[ 
	\mathbb L^2_{\mathbb K}(\mathbb I)=\left\{ P\in \mathbb L_{\mathbb K}(\mathbb I):  \mathbb E\left[\int_0^T|P_t|^2\,dt\right]<\infty \right\},
\]
\[
		\mathbb L^\infty_{\mathbb K}(\mathbb I)=\left\{ P\in\mathbb L_{\mathbb K}(\mathbb I):\|P\|_{\mathbb I,\infty}:= \esssup_{(\omega,t)\in\Omega\times\mathbb I}|P_t(\omega)|<\infty \right\},
\]
\begin{equation*}
	\begin{split}
\mathbb S^p_{\mathbb K}(\mathbb I)=\left\{ P\in \mathbb L_{\mathbb K}(\mathbb I): P\textrm{ has continuous trajectories and }\mathbb E\left[\sup_{t\in\mathbb I}|P_t|^p\right]<\infty \right\},\quad p=1,2,
	\end{split}
\end{equation*}
and
\begin{equation*}
\begin{split}
\mathbb S^\infty_{\mathbb K}(\mathbb I)=&~\left\{ P\in\mathbb L^\infty_{\mathbb K}(\mathbb I): P\textrm{ has continuous trajectories} \right\}.
\end{split}
\end{equation*}
Let $\mathcal T_{\mathbb K,\mathbb I}$ be the space of all-$\mathbb K$ stopping times valued in $\mathbb I$. The BMO space is defined as
\[
	\mathbb H^2_{\mathbb K,BMO}(\mathbb I)=\left\{ P\in\mathbb L_{\mathbb K}(\mathbb I):\|P\|_{\mathbb K,\mathbb I,BMO}:=\esssup_{\omega\in\Omega}\esssup_{\tau\in\mathcal T_{\mathbb K,\mathbb I}}\left(\mathbb E\left[\left.\int_{\mathbb I}|P_t|^2\,dt\right|\mathcal F_\tau\right]\right)^{1/2}<\infty \right\}
\]
 We will use the notation $\mathbb H^2_{\mathbb K,BMO,R}(\mathbb I)$ to denote the subspace of $\mathbb H^2_{\mathbb K,BMO}(\mathbb I)$ with all elements satisfying $\|\cdot\|_{\mathbb K,\mathbb I,BMO}\leq R$.
When $\mathbb I=[0,T]$, we will drop $\mathbb I$ in the notation of spaces and norms. 

Let $ L^2$ denote the space of all random variables that are squarely integrable. For the space of essentially bounded random variables $L^\infty$, we use $\|\cdot\|$ to denote the norm of essential supremum.

\section{$N$-Player Games}\label{sec:exp-N}
This section studies the solvability of the $N$-player game \eqref{wealth-exp-N-i}-\eqref{N-cost-i}. Firstly, we solve the multi-dimensional FBSDE \eqref{CN:exp-FBSDE-N} and then we verify the candidate \eqref{candidate-equilibrium-N} is indeed the Nash equilibrium of \eqref{wealth-exp-N-i}-\eqref{N-cost-i}. The following assumptions are in force in this section. 
\begin{assumption}[Assumptions of $N$-Player Games]\label{ass:N-player}
 Let $\mathscr{G}$ be a $\sigma$-algebra  that is independent of all Brownian motions.
\begin{itemize}
	\item[1.]  For each $i$, the tuple $(x^i,\theta^i,\alpha^i)$ is an independent $\mathscr{G}$ random variable, where $\theta^i$ is valued in $[0,1)$ and $\alpha^i$ is positively valued.
	\item [2.] The sequence $\{\overline b^i\}_i$ is uniformly bounded and progressively measurable w.r.t. $\mathbb F^i\vee\mathscr G$.
\end{itemize}
\end{assumption}
From now on, let $\mathscr G^{(N)}:=(\mathscr G^{(N)}_t)_{0\leq t\leq T}:=\mathbb F^{(N)}\vee\mathscr G$, where we recall $\mathbb F^{(N)}$ is the filtration generated by all Brownian motions $(W^1,\cdots,W^N,W^0)$.

\begin{remark}
	To solve the $N$-player game, a similar assumption on $\theta^i$ also appears in \cite{ET-2015,FR-2011}, where it is assumed that $\Pi_{i=1}^N\theta^i<1$. In order to show the convergence from the $N$-player game to MFG, \cite{ET-2015} further assumes $\theta^i=\theta^j<1$ for all $i\neq j$. \cite{FR-2011} provides a counterexample to show there might exist no equilibria when $\theta^i=1$. To solve the $N$-player game, our approach is to compare the multi-dimensional system and the mean field system, so we assume $\theta^i<1$ for all $i$, which is less general than $\Pi_{i=1}^n\theta^i<1$ but we allow $\theta^i$ to be heterogeneous across different players even in the MFG, like \cite{LZ-2019}. 
	
	 Moreover, in the $N$-player game, when $(x^i,\theta^i,\alpha^i)$ are random variables player $i$'s optimization is equivalent to the optimization when $(x^i,\theta^i,\alpha^i)$ are scalars. But this assumption makes a difference and is more general in MFG.  So we keep it from the very beginning for the simplicity of the statement.
\end{remark}

\subsection{Solvability of \eqref{CN:exp-FBSDE-N}}\label{sec:wellposedness-N-FBSDE}

Because of the possibly unbounded terminal condition $\theta^i\overline X^{-i}_T$, it is difficult to solve \eqref{CN:exp-FBSDE-N} directly. In the following we will transform the FBSDE \eqref{CN:exp-FBSDE-N} into a multi-dimensional quadratic BSDE with a bounded terminal condition. Note that a similar transformation also appears in \cite{ET-2015}, where the original equation is transformed into a BSDE with unbounded terminal condition while the terminal condition of our resulting BSDE is bounded. 

From \eqref{CN:exp-FBSDE-N} we have
\begin{equation}\label{CN:exp-N-multi-BSDE1}
\left\{\begin{split}
Y^i_t=&~\theta^i\overline{x}^{-i}+\frac{\theta^i}{N-1}\sum_{j\neq i}\int_0^T\left(\overline Z^{j}_s\overline b^j_s+\frac{|\overline b^j_s|^2}{\alpha^j}\right)\,ds+\frac{\theta^i}{N-1}\sum_{j\neq i}\int_0^T\left(Z^{jj}_s+\frac{b^j_s}{\alpha^j}\right)\,dW^j_s\\
&~+\frac{\theta^i}{N-1}\sum_{j\neq i}\int_0^T\left(Z^{j0}_s+\frac{b^{j0}_s}{\alpha^j}\right)\,dW^0_s\\
&~-\int_t^T\left(\overline Z^{i}_s\overline b^i_s+\frac{|\overline b^i_s|^2}{2\alpha^i}-\frac{\alpha^i}{2}\sum_{j\neq i}(Z^{ij}_s)^2\right)\,ds-\int_t^TZ^{ii}_s\,dW^i_s-\sum_{j\neq i}\int_t^TZ^{ij}_s\,dW^j_s,\\
&~-\int_t^TZ^{i0}_s\,dW^0_s\\
i=&~1,\cdots,N.
\end{split}\right.
\end{equation}
Now we divide the intergrals $\int_0^T\cdots\,ds$, $\int_0^T\cdots\,dW^j_s$ and $\int_0^T\cdots\,dW^0_s$ into two parts, $\int_0^t\cdots\,ds$, $\int_0^t\cdots\,dW^j_s$, $\int_0^t\cdots\,dW^0_s$ and $\int_t^T\cdots\,ds$, $\int_t^T\cdots\,dW^j_s$, $\int_t^T\cdots\,dW^0_s$. We make the latter part absorbed into the drift and volatility in \eqref{CN:exp-N-multi-BSDE1}, respectively. Define \begin{equation*}
	\begin{split}
\widetilde Y^i_\cdot:=&~Y^i_\cdot-\frac{\theta^i}{N-1}\sum_{j\neq i}\int_0^\cdot\left(\overline Z^{j}_s\overline b^j_s+\frac{|\overline b^j_s|^2}{\alpha^j}\right)\,ds-\frac{\theta^i}{N-1}\sum_{j\neq i}\int_0^\cdot\left(Z^{jj}_s+\frac{b^j_s}{\alpha^j}\right)\,dW^j_s\\
&~-\frac{\theta^i}{N-1}\sum_{j\neq i}\int_0^\cdot\left(Z^{j0}_s+\frac{b^{j0}_s}{\alpha^j}\right)\,dW^0_s.
	\end{split}
\end{equation*} Then $\widetilde Y^i$ satisfies 
\begin{equation}\label{CN:exp-N-multi-BSDE2}
\left\{\begin{split}
d\widetilde Y^i_t=&~\left(\overline Z^{i}_t \overline b^i_t+\frac{|\overline b^i_t|^2}{2\alpha^i}-\frac{\alpha^i}{2}\sum_{j\neq i}(Z^{ij}_t)^2-\frac{\theta^i}{N-1}\sum_{j\neq i}\left(\overline Z^{j}_t\overline b^j_t+\frac{|\overline b^j_t|^2}{\alpha^j}\right)\right)\,dt\\
&~+Z^{ii}_t\,dW^i_t+\sum_{j\neq i}\left(Z^{ij}_t-\frac{\theta^i}{N-1}\left(Z^{jj}_t+\frac{b^j_t}{\alpha^j}\right)\right)\,dW^j_t\\
&~+\left(Z^{i0}_t-\frac{\theta^i}{N-1}\sum_{j\neq i}\left(Z^{j0}_t+\frac{b^{j0}_t}{\alpha^j}\right)\right)\,dW^0_t,\\
\widetilde Y^i_T=&~\theta^i\overline{x}^{-i},\\
i=&~1,\cdots,N.
\end{split}\right.
\end{equation}
For each $i=1,\cdots,N$, let 
\begin{equation}\label{transformation-FB-to-tilde}
\widetilde Z^{ii}=Z^{ii},\quad \widetilde Z^{ij}=Z^{ij}-\frac{\theta^i}{N-1}\left(Z^{jj}+\frac{b^j}{\alpha^j}\right)\textrm{ for }j\neq i,\quad\widetilde Z^{i0}=Z^{i0}-\frac{\theta^i}{N-1}\sum_{j\neq i}\left(Z^{j0}+\frac{b^{j0}}{\alpha^j}\right).
\end{equation}
Let $C^j=Z^{j0}+\frac{b^{j0}}{\alpha^j}$ and $D^j=\widetilde Z^{j0}+\frac{b^{j0}}{\alpha^j}$. The last equality in \eqref{transformation-FB-to-tilde} implies
\[
	D^i=\left(1+\frac{\theta^i}{N-1}\right)C^i-\frac{\theta^i}{N-1}\sum_{j=1}^NC^j,
\]
which further implies
\[
	C^i=\frac{D^i}{1+\frac{\theta^i}{N-1}}+\frac{\frac{\theta^i}{N-1}\sum_{j=1}^NC^j}{1+\frac{\theta^i}{N-1}}.
\]
Taking sum from $1$ to $N$ and rearranging terms, one has
\[
	\left(1-\sum_{j=1}^N\frac{\theta^j}{N-1+\theta^j}\right)\sum_{j=1}^NC^j=\sum_{j=1}^N\frac{D^j}{1+\frac{\theta^j}{N-1}}.
\]
Since it is assumed that $0\leq \theta^j<1$ for all $j$ in Assumption \ref{ass:N-player}, it holds that $1-\sum_{j=1}^N\frac{\theta^j}{N-1+\theta^j}>0$ so one has
\[
	\sum_{j=1}^NC^j=\frac{\sum_{j=1}^N\frac{D^j}{1+\frac{\theta^j}{N-1}}}{	1-\sum_{j=1}^N\frac{\theta^j}{N-1+\theta^j}}.
\]
Thus, we get 
\[
	C^i=\frac{ D^i+\frac{\theta^i}{N-1}  \frac{\sum_{j=1}^N\frac{D^j}{1+\frac{\theta^j}{N-1}}}{1-\sum_{j=1}^N\frac{\theta^j}{N-1+\theta^j}} }{1+\frac{\theta^i}{N-1}},
\]
which implies that
	\[
	Z^{i0}=\frac{\widetilde Z^{i0}+\frac{\theta^i}{N-1}  \frac{\sum_{j=1}^N\frac{\widetilde Z^{j0}}{1+\frac{\theta^j}{N-1}}}{1-\sum_{j=1}^N\frac{\theta^j}{N-1+\theta^j}}+\frac{\theta^i}{N-1}  \frac{\sum_{j=1}^N\frac{\frac{b^{j0}}{\alpha^j}}{1+\frac{\theta^j}{N-1}}}{1-\sum_{j=1}^N\frac{\theta^j}{N-1+\theta^j}}-\frac{\theta^i}{N-1}\frac{b^{i0}}{\alpha^i}}{1+\frac{\theta^i}{N-1}}:=f_{i,N}(\cdot,\widetilde Z^{i0};\widetilde Z^{k0},k\neq i).
	\]
Then, the BSDEs \eqref{CN:exp-N-multi-BSDE2} are transformed into
\begin{equation}\label{CN:exp-N-multi-BSDE3}
\left\{\begin{split}
d\widetilde Y^i_t
=&~\left\{\widetilde Z^{ii}_t b^i_t+b^{i0}_tf_{i,N}(t,\widetilde Z^{i0}_t;\widetilde Z^{k0}_t,k\neq i)+\frac{|\overline b^i_t|^2}{2\alpha^i}-\frac{\alpha^i}{2}\sum_{j\neq i}\left(\widetilde Z^{ij}_t+\frac{\theta^i}{N-1}\left(\widetilde Z^{jj}_t+\frac{b^j_t}{\alpha^j}\right)\right)^2\right.\\
&~\left.-\frac{\theta^i}{N-1}\sum_{j\neq i}\left(b^j_t\widetilde Z^{jj}_t+b^{j0}_tf_{j,N}(t,\widetilde Z^{j0}_t;\widetilde Z^{k0}_t,k\neq j)\right)-\frac{\theta^i}{N-1}\sum_{j\neq i}\frac{|\overline b^j_t|^2}{\alpha^j}\right\}\,dt\\
&~+\widetilde Z^{ii}_t\,dW^i_t+\sum_{j\neq i}\widetilde Z^{ij}_t\,dW^j_t+\widetilde Z^{i0}_t\,dW^0_t\\
\widetilde Y^i_T=&~\theta^i\overline{x}^{-i},\\
i=&~1,\cdots,N.
\end{split}\right.
\end{equation}
In this way, we transform the multi-dimensional FBSDE with a possibly unbounded terminal condition \eqref{CN:exp-FBSDE-N} to the multi-dimensional BSDE with a bounded terminal condition \eqref{CN:exp-N-multi-BSDE3}. Moreover, the solutions to \eqref{CN:exp-FBSDE-N} and \eqref{CN:exp-N-multi-BSDE3} have one to one correspondence. 
To solve \eqref{CN:exp-N-multi-BSDE3}, we first consider the following sequence of benchmark conditional mean field BSDEs
\begin{equation}\label{CN:exp-N-benchmark-BSDE}
\begin{split}
d\breve Y^i_t=&~\left\{b^i_t\breve Z^i_t+\frac{|\overline b^i_t|^2}{2\alpha^i}+b^{i0}_t\left(\breve Z^{i0}_t+\frac{\theta^i}{1-\mathbb E[\theta^i]}\mathbb E[\breve Z^{i0}_t|\mathcal F^0_t]+\frac{\theta^i}{1-\mathbb E[\theta^i]}\mathbb E\left[\left.\frac{b^{i0}_t}{\alpha^i}\right|\mathcal F^0_t\right]\right)\right.\\
&~\left.-\theta^i\mathbb E\left[\left.\frac{|\overline b^i_t|^2}{\alpha^i}\right|\mathcal F^0_t\right]-\theta^i\mathbb E[b^i_t\breve Z^i_t|\mathcal F^0_t]-\theta^i\mathbb E[b^{i0}_t\breve Z^{i0}_t|\mathcal F^0_t]-\frac{\theta^i}{1-\mathbb E[\theta^i]}\mathbb E[\theta^ib^{i0}_t|\mathcal F^0_t]\mathbb E[\breve Z^{i0}_t|\mathcal F^0_t]\right.\\
&~\left.-\frac{\theta^i}{1-\mathbb E[\theta^i]}\mathbb E[\theta^ib^{i0}_t|\mathcal F^0_t]\mathbb E\left[\left.\frac{b^{i0}_t}{\alpha^i}\right|\mathcal F^0_t\right]\right\}\,dt+\breve Z^i_t\,dW^i_t+\breve Z^{i0}_t\,dW^0_t,\quad \breve Y^i_T=\theta^i\overline{x}^{-i}\\
i=&~1,\cdots,N.
\end{split}
\end{equation}
Due to the Lipschitz property of the driver, the well-posedness of \eqref{CN:exp-N-benchmark-BSDE} in $\mathbb S^2\times\mathbb H^2$ is obvious. However, the essential boundedness of $\breve Y^i$ and the BMO property of $\breve Z^{i}$ and $\breve Z^{i0}$ are necessary for the well-posedness of \eqref{CN:exp-N-multi-BSDE3} as well as the convergence to the MFG. This is done by the next lemma.
\begin{lemma}\label{lem:benchmark-BSDE}
	Under Assumption \ref{ass:N-player}, for fixed $i$, for each $R>0$, choosing $\theta^i$ and $\overline b^i$ small enough such that
	\begin{equation}\label{ass-theta-benchmark-BSDE}
	\left\{\begin{split}
	&~10\|b^i\|^2_{\mathscr G^{(N)},BMO}\leq \frac{1}{8},\quad 20\left(1+\frac{\|\theta^i\|^2}{(1-\mathbb E[\theta^i])^2}\right)\|b^{i0}\|^2_{\mathscr G^{(N)},BMO}\leq \frac{1}{8},\\
	&~ \frac{5}{2}\left\|\frac{\overline b^i}{\sqrt{\alpha^i}}\right\|^4_{\mathscr G^{(N)},BMO}\leq \frac{3R^2}{40},\quad\frac{40\|\theta^i\|^2}{(1-\mathbb E[\theta^i])^2}\|b^{i0}\|_{\mathscr G^{(N)},BMO}^2\left\|\frac{b^{i0}}{\alpha^i}\right\|_{\mathscr G^{(N)},BMO}^2\leq \frac{3R^2}{40},\\
	&~ 10\|\theta^i\|^2\left\|\frac{\overline b^i}{\sqrt{\alpha^i}}\right\|^4_{\mathscr G^{(N)},BMO}\leq \frac{3R^2}{40},\quad  \|\theta^i\overline{x}^{-i}\|^2 \leq \frac{3R^2}{40},
	\end{split}\right.
	\end{equation}
	 there exists a unique solution of \eqref{CN:exp-N-benchmark-BSDE} such that
	 \begin{equation}\label{estimate-MF-BSDE-R}
	 	\|\breve Y^i\|_\infty^2+\|\breve Z^i\|^2_{\mathscr G^{(N)},BMO}+\|\breve Z^{i0}\|^2_{\mathscr G^{(N)},BMO}\leq R^2.
	 \end{equation}

\end{lemma}
\begin{proof}
The well-posedness of \eqref{CN:exp-N-benchmark-BSDE} in $\mathbb S^2\times \mathbb H^2$ holds because of the Lipschitz property of the driver in \eqref{CN:exp-N-benchmark-BSDE}. 
	It remains to prove $(\breve Y^i,\breve Z^i,\breve Z^{i0})\in\mathbb S^\infty_{\mathscr G^{(N)}}\times\mathbb H^2_{\mathscr G^{(N)},BMO,R}\times\mathbb H^2_{\mathscr G^{(N)},BMO,R}$.

	By It\^o's formula, taking conditional expectations and Proposition \ref{prop:appendix} we have for each $\tau\in\mathcal T_{\mathscr G^{(N)}}$
	\begin{equation}\label{CN:inequ:before-BMO}
	\begin{split}
	&~(\breve Y^i_\tau)^2+\mathbb E\left[\left.\int_\tau^T(\breve Z^i_s)^2\,ds\right|\mathscr G^{(N)}_\tau\right]+\mathbb E\left[\left.\int_\tau^T(\breve Z^{i0}_s)^2\,ds\right|\mathscr G^{(N)}_\tau\right]\\
	\leq&~2\|\breve Y^i\|_\infty\mathbb E\left[\left.\int_\tau^T\left|b^i_s\breve Z^i_s+\frac{|\overline b^i_s|^2}{2\alpha^i}+b^{i0}_s\left(\breve Z^{i0}_s+\frac{\theta^i\mathbb E[\breve Z^{i0}_s|\mathcal F^0_s]}{1-\mathbb E[\theta^i]}+\frac{\theta^i}{1-\mathbb E[\theta^i]}\mathbb E\left[\left.\frac{b^{i0}_t}{\alpha^i}\right|\mathcal F^0_s\right]\right)\right|\,ds\right|\mathscr G^{(N)}_\tau\right]\\
	&~+2\|\breve Y^i\|_\infty\mathbb E\left[\left.\int_\tau^T\left(\theta^i\mathbb E\left[\left.\frac{|\overline b^i_s|^2}{\alpha^i}\right|\mathcal F^0_s\right]+\theta^i\mathbb E[b^i_s\breve Z^i_s|\mathcal F^0_s]\right.\right.\right.\\
	&~\left.\left.\left.+\theta^i\mathbb E[b^{i0}_s\breve Z^{i0}_s|\mathcal F^0_s]+\frac{\theta^i}{1-\mathbb E[\theta^i]}\mathbb E[\theta^ib^{i0}_s|\mathcal F^0_s]\mathbb E[\breve Z^{i0}_s|\mathcal F^0_s]\right)\,ds\right|\mathscr G^{(N)}_\tau\right]\\
	&~+2\|\breve Y^i\|_\infty\frac{\theta^i}{1-\mathbb E[\theta^i]}\mathbb E\left[\left.\int_\tau^T\mathbb E[\theta^ib^{i0}_s|\mathcal F^0_s]\mathbb E\left[\left.\frac{b^{i0}_s}{\alpha^i}\right|\mathcal F^0_s\right]\,ds\right|\mathscr G^{(N)}_\tau\right]+(\theta^i)^2(\overline{x}^{-i})^2.
	\end{split}
	\end{equation}
By Lemma \ref{lem:app-estimate-BMO}, taking supremum on both sides of \eqref{CN:inequ:before-BMO} and using Young's inequality $2ab\leq\frac{a^2}{10}+10b^2$, we have
	\begin{align*}
	&~\|\breve Y^i\|^2_\infty+\|\breve Z^i\|^2_{\mathscr G^{(N)},BMO}+\|\breve Z^{i0}\|^2_{\mathscr G^{(N)},BMO}\\
	\leq&~\frac{7}{10}\|\breve Y^i\|_\infty^2+10\|b^i\|_{\mathscr G^{(N)},BMO}^2\|\breve Z^i\|^2_{\mathscr  G^{(N)},BMO}+\frac{5}{2}\left\|\frac{\overline b^i}{\sqrt{\alpha^i}}\right\|^4_{\mathscr G^{(N)},BMO}\\
	&~+20\left(1+\frac{\|\theta^i\|^2}{(1-\mathbb E[\theta^i])^2}\right)\|b^{i0}\|^2_{\mathscr G^{(N)},BMO}\|\breve Z^{i0}\|^2_{\mathscr G^{(N)},BMO}+\frac{40\|\theta^i\|^2}{(1-\mathbb E[\theta^i])^2}\|b^{i0}\|_{\mathscr G^{(N)},BMO}^2\left\|\frac{b^{i0}}{\alpha^i}\right\|_{\mathscr G^{(N)},BMO}^2\\
	&~+10\|\theta^i\|^2\left\|\frac{\overline b^i}{\sqrt{\alpha^i}}\right\|^4_{\mathscr G^{(N)},BMO}+10\|\theta^i\|^2\|b^i\|^2_{\mathscr G^{(N)},BMO}\|\breve Z^i\|^2_{\mathscr G^{(N)},BMO}\\
	&~+10\left(\|\theta^i\|+\frac{\|\theta^i\|^2}{1-\mathbb E[\theta^i]}\right)^2\|b^{i0}\|^2_{\mathscr G^{(N)},BMO}\|\breve Z^{i0}\|^2_{\mathscr G^{(N)},BMO}+\|\theta^i\overline{x}^{-i}\|^2.
	\end{align*}
By the first line of \eqref{ass-theta-benchmark-BSDE}	it holds that
	\begin{equation*}\label{CN:estimate-benchmark}
	\begin{split}
		&~\|\breve Y^i\|^2_{\infty}+\|\breve Z^i\|^2_{\mathscr G^{(N)},BMO}+\|\breve Z^{i0}\|^2_{\mathscr G^{(N)},BMO}\\
		\leq&~ \frac{10}{3}\left(\frac{5}{2}\left\|\frac{\overline b^i}{\sqrt{\alpha^i}}\right\|^4_{\mathscr G^{(N)},BMO}+\frac{40\|\theta^i\|^2}{(1-\mathbb E[\theta^i])^2}\|b^{i0}\|_{\mathscr G^{(N)},BMO}^2\left\|\frac{b^{i0}}{\alpha^i}\right\|_{\mathscr G^{(N)},BMO}^2\right.\\
		&~\left.+10\|\theta^i\|^2\left\|\frac{\overline b^i}{\sqrt{\alpha^i}}\right\|^4_{\mathscr G^{(N)},BMO}+\|\theta^i\overline{x}^{-i}\|^2\right),
	\end{split}
	\end{equation*}
	from which the desired result follows due to the second and the third lines of \eqref{ass-theta-benchmark-BSDE}.	
\end{proof}

\begin{remark} Although $(\breve Y^i,\breve Z^i,\breve Z^{i0})$ are progressively measurable w.r.t. $\mathbb F^i\vee\mathscr G$, in Lemma \ref{lem:benchmark-BSDE} we use the larger filtration $\mathscr G^{(N)}$ in order to use this lemma to prove the well-posedness result of \eqref{CN:exp-N-multi-BSDE3}.
	
\end{remark}

With the well-posedness of \eqref{CN:exp-N-benchmark-BSDE} especially the BMO bound of the solution at hand, we can solve \eqref{CN:exp-N-multi-BSDE3}. The approach is to compare \eqref{CN:exp-N-multi-BSDE3} with the  system of benchmark BSDEs \eqref{CN:exp-N-benchmark-BSDE}. Let 
\begin{equation}\label{transformation-tilde-to-hat}
	\widehat Y^i=\widetilde Y^i-\breve Y^i,\quad \widehat Z^{ii}=\widetilde Z^{ii}-\breve Z^i,\quad\widehat Z^{i0}=\widetilde Z^{i0}-\breve Z^{i0}
\end{equation}
 for each $i=1,\cdots,N$ and there holds that
\begin{equation}\label{CN:exp-N-multi-BSDE4}
\left\{\begin{split}
d\widehat Y^i_t=&~\left(b^i_t\widehat Z^{ii}_t-\frac{\alpha^i}{2}\sum_{j\neq i}\left(\widetilde Z^{ij}_t+\theta^i\frac{\widehat Z^{jj}_t+\breve Z^j_t+\frac{b^j_t}{\alpha^j}}{N-1}\right)^2-\frac{\theta^i}{N-1}\sum_{j\neq i}b^j_t\widehat Z^{jj}_t\right)\,dt\\
&~+\left(b^{i0}_t\frac{\widehat Z^{i0}_t}{1+\frac{\theta^i}{N-1}}+\frac{\theta^i b^{i0}_t}{(1+
	\frac{\theta^i}{N-1})}\frac{1}{N-1}\frac{\sum_{j=1}^N\frac{\widehat Z^{j0}_t}{1+\frac{\theta^j}{N-1}}}{1-\sum_{j=1}^N\frac{\theta^j}{N-1+\theta^j}}-\frac{\theta^i}{N-1}\sum_{j\neq i}\frac{b^{j0}_t \widehat Z^{j0}_t}{1+\frac{\theta^j}{N-1}}\right)\,dt\\
&~+\left(-\frac{\theta^i}{N-1}\sum_{j\neq i}b^{j0}_t\frac{\frac{\theta^j}{N-1}\frac{\sum_{k=1}^N\frac{\widehat Z^{k0}_t}{1+\frac{\theta^k}{N-1}}}{1-\sum_{k=1}^N\frac{\theta^k}{N-1+\theta^k}}}{1+\frac{\theta^j}{N-1}}+B^{i,N}_t\right)\,dt\\
&~+\widehat Z^{ii}_t\,dW^i_t+\widehat Z^{i0}_t\,dW^0_t+\sum_{j\neq i}\widetilde Z^{ij}_t\,dW^j_t\\
\widehat Y^i_T=&~0,\\
i=&~1,\cdots,N.
\end{split}\right.
\end{equation}
where
\begin{align*}
B^{i,N}_t:=&~-\frac{\theta^i}{N-1}\sum_{j\neq i}b^j_t\breve Z^j_t+\theta^i\mathbb E[b^i_t\breve Z^i_t|\mathcal F^0_t]\\
&~+\left(\frac{b^{i0}_t\frac{\theta^i}{N-1}\frac{\sum_{k=1}^N\frac{\breve Z^{k0}_t}{1+\frac{\theta^k}{N-1}}}{1-\sum_{k=1}^N\frac{\theta^k}{N-1+\theta^k}}}{1+\frac{\theta^i}{N-1}}-\frac{\theta^i b^{i0}_t}{1-\mathbb E[\theta^i]}\mathbb E[\breve Z^{i0}_t|\mathcal F^0_t]\right)\\
&~+\left(\frac{b^{i0}_t\frac{\theta^i}{N-1}\frac{\sum_{k=1}^N\frac{\frac{b^{k0}_t}{\alpha^k}}{1+\frac{\theta^k}{N-1}}}{1-\sum_{k=1}^N\frac{\theta^k}{N-1+\theta^k}}}{1+\frac{\theta^i}{N-1}}-\frac{\theta^i b^{i0}_t}{1-\mathbb E[\theta^i]}\mathbb E\left[\left.\frac{b^{i0}_t}{\alpha^i}\right|\mathcal F^0_t\right]\right)-\frac{\theta^i}{N-1}\sum_{j\neq i}\frac{b^{j0}_t\breve Z^{j0}_t}{1+\frac{\theta^j}{N-1}}+\theta^i \mathbb E[b^{i0}_t\breve Z^{i0}_t|\mathcal F^0_t]\\
&~+\left(-\frac{\theta^i}{N-1}\sum_{j\neq i}b^{j0}_t\frac{\frac{\theta^j}{N-1}\frac{\sum_{k=1}^N\frac{\breve Z^{k0}_t}{1+\frac{\theta^k}{N-1}}}{1-\sum_{k=1}^N\frac{\theta^k}{N-1+\theta^k}}}{1+\frac{\theta^j}{N-1}} +\frac{\theta^i}{1-\mathbb E[\theta^i]}\mathbb E[\theta^ib^{i0}_t|\mathcal F^0_t]\mathbb E[\breve Z^{i0}_t|\mathcal F^0_t]\right)\\
&~+\left(-\frac{\theta^i}{N-1}\sum_{j\neq i}b^{j0}_t\frac{\frac{\theta^j}{N-1}\frac{\sum_{k=1}^N\frac{\frac{b^{k0}_t}{\alpha^k}}{1+\frac{\theta^k}{N-1}}}{1-\sum_{k=1}^N\frac{\theta^k}{N-1+\theta^k}}}{1+\frac{\theta^j}{N-1}}+\frac{\theta^i}{1-\mathbb E[\theta^i]}\mathbb E[\theta^ib^{i0}_t|\mathcal F^0_t]\mathbb E\left[\left.\frac{b^{i0}_t}{\alpha^i}\right|\mathcal F^0_t\right]\right)\\
&~+\theta^i\left(-\frac{1}{N-1}\sum_{j\neq i}\frac{|\overline b^j_t|^2}{\alpha^j}+\mathbb E\left[\left.\frac{|\overline b^i_t|^2}{\alpha^i}\right|\mathcal F^0_t\right]\right)\\
&~-\frac{\theta^i}{(N-1)\left(1+\frac{\theta^i}{N-1}\right)}\frac{(b^{i0}_t)^2}{\alpha^i}+\frac{\theta^i}{(N-1)^2}\sum_{j\neq i}\frac{\theta^j(b^{j0}_t)^2}{\alpha^j\left(1+\frac{\theta^j}{N-1}\right)}\\
&~- \frac{\theta^i}{N-1+\theta^i  }b^{i0}\breve Z^{i0}.
\end{align*}
The estimate of $B^{i,N}$ can be found in Lemma \ref{lem:estimate-B}.
In particular, $\|\sqrt{|B^{i,N}|}\|_{\mathscr G^{(N)},BMO}$ can be made small enough by assuming $\{\|\theta^i\|\}_i$ small enough.

In order to facilitate the presentation of the estimate, we introduce the following notation: for two sequences of stochastic processes $f^n$ and $g^n$, denote
\begin{equation*}
\left\{\begin{split}
&~\overline{\|f^\cdot\|   }^{(-i),N}:= \frac{1}{N-1}\sum_{j\neq i}\|f^j\|_{\mathscr G^{(N)},BMO},\\
&~ \overline{ \| f^\cdot,g^\cdot \|  }^{(-i),N}:=\frac{1}{N-1}\sum_{j\neq i}\|f^j\|_{\mathscr G^{(N)},BMO}\|g^j\|_{\mathscr G^{(N)},BMO},\\
&~ \overline{\|f^\cdot\|  }^N :=\frac{1}{N}\sum_{j=1}^N\|f^j \|_{\mathscr G^{(N)},BMO}.
\end{split}\right.
\end{equation*}

The following theorem establishes the well-posedness result of \eqref{CN:exp-N-multi-BSDE4} and thus \eqref{CN:exp-N-multi-BSDE3}.
\begin{theorem}\label{thm:CN-wellposedness-exp-multi-BSDE}
Let  
\[
A^{i,N}=\sqrt{2\left\{ 20\left\|\sqrt{|B^{i,N}|}\right\|_{\mathscr G^{(N)},BMO}^2+\frac{640\|\alpha^i\|^2\|\theta^i\|^4R^4}{(N-1)^2} +\frac{80\|\alpha^i\|^2\|\theta^i\|^4}{(N-1)^2}\left(\overline{\left\|\frac{b^\cdot}{\alpha^\cdot}\right\|^2}^{(-i),N}\right)^2\right\}}.
\]
In addition to Assumption \ref{ass:N-player} and the assumption \eqref{ass-theta-benchmark-BSDE}, for each fixed $N$ let $\{b^i\}$, $R$ and $\{\theta^i\}$ small enough such that
\begin{equation}\label{ass-multi-BSDE}
\left\{\begin{split}
&~20\|b^i\|^2_{\mathscr G^{(N)},BMO}\leq \frac{1}{20},\quad 640\|\alpha^i\|^2(\max_iA^{i,N})^2\leq \frac{1}{20},\quad \\
&~20\|\theta^i\|^2\overline{\|b^\cdot\|^2}^{(-i),N}\leq \frac{1}{20},\quad 20\|b^{i0}\|^2_{\mathscr G^{(N)},BMO}\leq \frac{1}{20},\\
&~\frac{20\|\theta^i\|^2}{\left(1-\sum_{j=1}^N\frac{\|\theta^j\|}{N-1+\|\theta^j\|}\right)^2}\|b^{i0}\|_{\mathscr G^{(N)},BMO}^2\leq \frac{1}{20},\quad 20\|\theta^i\|^2\overline{\|b^{\cdot0}\|^2}^{(-i),N}\leq \frac{1}{20},\\
&~\quad\frac{20\|\theta^i\|^2}{\left(1-\sum_{j=1}^N\frac{\|\theta^j\|}{N-1+\|\theta^j\|}\right)^2}\left(\overline{\|\theta^\cdot b^{\cdot0}\|}^{(-i),N}\right)^2\leq \frac{1}{20}.
\end{split}\right.
\end{equation}
and
\begin{equation}\label{ass-multi-BSDE2}
40\|\alpha^i\|^2\left\{ 6(\max_iA^{(i,N)})^2+12\left(2\frac{(\max_iA^{i,N})^2}{N-1}+4\frac{R^2}{N-1}+\frac{4}{N-1}\overline{\left\|\frac{b^\cdot}{\alpha^\cdot}\right\|^2}^{(-i),N}\right) \right\}\leq\frac{1}{20},
\end{equation}
then there exists a unique solution to \eqref{CN:exp-N-multi-BSDE4} and thus \eqref{CN:exp-N-multi-BSDE3}. In particular, the solution to \eqref{CN:exp-N-multi-BSDE4} admits the following estimate
\begin{equation}\label{estimate-thm-widehat}
\frac{1}{2}\|\widehat Y^i\|_{\infty}^2+\|\widehat Z^{ii}\|^2_{\mathscr G^{(N)},BMO}+\|\widehat Z^{i0}\|^2_{\mathscr G^{(N)},BMO}+\sum_{j\neq i}\|\widetilde Z^{ij}\|^2_{\mathscr G^{(N)},BMO}\leq (\max_iA^{i,N})^2.
\end{equation}
As a corollary, the well-posedness of \eqref{CN:exp-N-multi-BSDE2} and \eqref{CN:exp-FBSDE-N} holds. That is, there exists a uniqe $(\widetilde Y^i,\widetilde Z^{ii},\widetilde Z^{i0},\widetilde Z^{ij},j\neq i)\in\mathbb S^\infty_{\mathscr G^{(N)}}\times\mathbb H^2_{\mathscr G^{(N)},BMO}\times\cdots\times\mathbb H^2_{\mathscr G^{(N)},BMO}$, $i=1,\cdots,N$, such that \eqref{CN:exp-N-multi-BSDE2} holds, and there exists a unique $(X^i,Y^i,Z^{ii},Z^{i0},Z^{ij},j\neq i)\in \mathbb S^2_{\mathscr G^{(N)}}\times \mathbb S^2_{\mathscr G^{(N)}}\times\mathbb H^2_{\mathscr G^{(N)},BMO}\times\cdots\times\mathbb H^2_{\mathscr G^{(N)},BMO}$, for each $i=1,\cdots,N$, such that \eqref{CN:exp-FBSDE-N} holds.
\end{theorem}
\begin{proof}
We use the fixed point arguement in \cite{Tevzadze2008}. 
For each $i=1,\cdots,N$, fix $((\widetilde z^{ij})_{j\neq i},\widehat z^{ii},\widehat z^{i0})$ such that \begin{equation}\label{wellposedness-mapping-1}
\sum_{j\neq i}\|\widetilde z^{ij}\|^2_{\mathscr G^{(N)},BMO}+\|\widehat z^{ii}\|^2_{\mathscr G^{(N)},BMO}+\|\widehat z^{i0}\|^2_{\mathscr G^{(N)},BMO}\leq (\max_iA^{i,N})^2
\end{equation}
 and consider the resulting BSDEs
\begin{equation}\label{CN:exp-N-multi-BSDE6}
\left\{\begin{split}
d\widehat Y^i_t=&~\left(b^i_t\widehat z^{ii}_t-\frac{\alpha^i}{2}\sum_{j\neq i}\left(\widetilde z^{ij}_t+\theta^i\frac{\widehat z^{jj}_t+\breve Z^j_t+\frac{b^j_t}{\alpha^j}}{N-1}\right)^2-\frac{\theta^i}{N-1}\sum_{j\neq i}b^j_t\widehat z^{jj}_t\right)\,dt\\
&~+\left(b^{i0}_t\frac{\widehat z^{i0}_t}{1+\frac{\theta^i}{N-1}}+\frac{\theta^i b^{i0}_t}{(1+
	\frac{\theta^i}{N-1})}\frac{1}{N-1}\frac{\sum_{j=1}^N\frac{\widehat z^{j0}_t}{1+\frac{\theta^j}{N-1}}}{1-\sum_{j=1}^N\frac{\theta^j}{N-1+\theta^j}}-\frac{\theta^i}{N-1}\sum_{j\neq i}\frac{b^{j0}_t \widehat z^{j0}_t}{1+\frac{\theta^j}{N-1}}\right)\,dt\\
&~+\left(-\frac{\theta^i}{N-1}\sum_{j\neq i}b^{j0}_t\frac{\frac{\theta^j}{N-1}\frac{\sum_{k=1}^N\frac{\widehat z^{k0}_t}{1+\frac{\theta^k}{N-1}}}{1-\sum_{k=1}^N\frac{\theta^k}{N-1+\theta^k}}}{1+\frac{\theta^j}{N-1}}+B^{i,N}_t\right)\,dt\\
&~+\widehat Z^{ii}_t\,dW^i_t+\widehat Z^{i0}_t\,dW^0_t+\sum_{j\neq i}\widetilde Z^{ij}_t\,dW^j_t\\
\widehat Y^i_T=&~0,\\
i=&~1,\cdots,N.
\end{split}\right.
\end{equation}
	By Proposition \ref{prop:appendix}, there exists a unique solution $(\widehat Y^i,\widehat Z^{ii},\widehat Z^{i0},\widetilde Z^{ij},j\neq i)\in\mathbb S^\infty_{\mathscr G^{(N)}}\times \mathbb H^2_{\mathscr G^{(N)},BMO}\times\cdots\times\mathbb H^2_{\mathscr G^{(N)},BMO}$ for each $i=1,\cdots,N$. In the following we will prove 
	$\sum_{j\neq i}\|\widetilde Z^{ij}\|^2_{\mathscr G^{(N)},BMO}+\|\widehat Z^{ii}\|^2_{\mathscr G^{(N)},BMO}+\|\widehat Z^{i0}\|^2_{\mathscr G^{(N)},BMO}\leq (\max_iA^{i,N})^2$
	 for each $i$ and the mapping $\{(\widehat z^{ii},\widehat z^{i0},\widetilde z^{ij},j\neq i)\}_{i=1,\cdots,N}\mapsto \{(\widehat Z^{ii},\widehat Z^{i0},\widetilde Z^{ij},j\neq i)\}_{i=1,\cdots,N}$ is a contraction. By It\^o's formula, Lemma \ref{lem:app-estimate-BMO} and Young's inequality $2ab\leq \frac{a^2}{20}+20b^2$, we have for each $\tau\in\mathcal T_{\mathscr G^{(N)}}$
	\begin{align*}
				&~(\widehat Y^i_\tau)^2+\mathbb E \left[\left.\int_\tau^T(\widehat Z^{ii}_s)^2\,ds\right|\mathscr G^{(N)}_\tau\right]+\mathbb E \left[\left.\int_\tau^T(\widehat Z^{i0}_s)^2\,ds\right|\mathscr G^{(N)}_\tau\right]+\sum_{j\neq i}\mathbb E\left[\left.\int_\tau^T(\widetilde Z^{ij}_s)^2\,ds\right|\mathscr G^{(N)}_\tau\right]\\
				\leq&~\frac{1}{20}\|\widehat Y^i\|^2_\infty+20\|b^i\|^2_{\mathscr G^{(N)},BMO}\|\widehat z^{ii}\|^2_{\mathscr G^{(N)},BMO}+\frac{1}{20}\|\widehat Y^i\|^2_\infty+20\|\alpha^i\|^2\left(\sum_{j\neq i}\|\widetilde z^{ij}\|^2_{\mathscr G^{(N)},BMO}\right)^2+\frac{1}{20}\|\widehat Y^i\|^2_\infty\\
				&~+\frac{80\|\alpha^i\|^2\|\theta^i\|^4}{(N-1)^2}\left(\overline{\|\widehat z^{\cdot\cdot}+\breve Z^\cdot\|^2}^{(-i),N}\right)^2+\frac{1}{20}\|\widehat Y^i\|^2_\infty+\frac{80\|\alpha^i\|^2\|\theta^i\|^4}{(N-1)^2}\left(\overline{\left\|\frac{b^\cdot}{\alpha^\cdot}\right\|^2}^{(-i),N}\right)^2\\
				&~+\frac{1}{20}\|\widehat Y^i\|^2_\infty+20\|\theta^i\|^2\overline{\|b^\cdot\|^2}^{(-i),N}\overline{\|\widehat z^{\cdot\cdot}\|^2}^{(-i),N}+\frac{1}{20}\|\widehat Y^i\|^2_\infty+20\|b^{i0}\|^2_{\mathscr G^{(N)},BMO}\|\widehat z^{i0}\|^2_{\mathscr G^{(N)},BMO}\\
				&~+\frac{1}{20}\|\widehat Y^i\|_{\infty}^2+\frac{20\|\theta^i\|^2}{\left(1-\sum_{j=1}^N\frac{\|\theta^j\|}{N-1+\|\theta^j\|}\right)^2}\|b^{i0}\|^2_{\mathscr G^{(N)},BMO}\left(\frac{N}{N-1}\overline{\|\widehat z^{\cdot0}\|}^N\right)^2\\
				&~+\frac{1}{20}\|\widehat Y^i\|_{\infty}^2+20\|\theta^i\|^2\frac{N}{N-1}\overline{\|\widehat z^{\cdot0}\|^2}^N\overline{\|b^{\cdot0}\|^2}^{(-i),N}+\frac{1}{20}\|\widehat Y^i\|_{\infty}^2\\
				&~+\frac{20\|\theta^i\|^2}{\left(1-\sum_{j=1}^N\frac{\|\theta^j\|}{N-1+\|\theta^j\|}\right)^2}\left(\overline{\|\theta^\cdot b^{\cdot0}\|}^{(-i),N}\right)^2\left(\frac{N}{N-1}\overline{\|\widehat z^{\cdot0}\|}^{N}\right)^2+\frac{1}{20}\|\widehat Y^i\|_{\infty}^2+20\left\|\sqrt{|B^{i,N}|}\right\|_{\mathscr G^{(N)},BMO}^2.
	\end{align*}
By the estimate above, \eqref{wellposedness-mapping-1} and \eqref{estimate-MF-BSDE-R}, and rearranging terms, we have
\begin{align*}
	&~\frac{1}{2}\|\widehat Y^i\|_{\infty}^2+\|\widehat Z^{ii}\|^2_{\mathscr G^{(N)},BMO}+\|\widehat Z^{i0}\|^2_{\mathscr G^{(N)},BMO}+\sum_{j\neq i}\|\widetilde Z^{ij}\|^2_{\mathscr G^{(N)},BMO}\\
	\leq&~20\|b^i\|^2_{\mathscr G^{(N)},BMO}(\max_iA^{i,N})^2+20\|\alpha^i\|^2(\max_iA^{i,N})^4+\frac{320\|\alpha^i\|^2\|\theta^i\|^4}{(N-1)^2}\left((\max_iA^{i,N})^2+R^2\right)^2\\
	&~+\frac{80\|\alpha^i\|^2\|\theta^i\|^4}{(N-1)^2}\left(\overline{\left\|\frac{b^\cdot}{\alpha^\cdot}\right\|^2}^{(-i),N}\right)^2
	+20\|\theta^i\|^2\overline{\|b^\cdot\|^2}^{(-i),N}(\max_iA^{i,N})^2\\
	&~+20\|b^{i0}\|^2_{\mathscr G^{(N)},BMO}(\max_iA^{i,N})^2+\frac{20\|\theta^i\|^2}{\left(1-\sum_{j=1}^N\frac{\|\theta^j\|}{N-1+\|\theta^j\|}\right)^2}\|b^{i0}\|^2_{\mathscr G^{(N)},BMO}(\max_iA^{i,N})^2\\
	&~+20\|\theta^i\|^2(\max_iA^{i,N})^2\overline{\|b^{\cdot0}\|^2}^{(-i),N}+\frac{20\|\theta^i\|^2}{\left(1-\sum_{j=1}^N\frac{\|\theta^j\|}{N-1+\|\theta^j\|}\right)^2}\left(\overline{\|\theta^\cdot b^{\cdot0}\|}^{(-i),N}\right)^2(\max_iA^{i,N})^2\\
	&~+20\left\|\sqrt{|B^{i,N}|}\right\|_{\mathscr G^{(N)},BMO}^2.
\end{align*}
Thus, by \eqref{ass-multi-BSDE} it holds that
	\begin{equation*}
	\frac{1}{2}\|\widehat Y^i\|_{\infty}^2+\|\widehat Z^{ii}\|^2_{\mathscr G^{(N)},BMO}+\|\widehat Z^{i0}\|^2_{\mathscr G^{(N)},BMO}+\sum_{j\neq i}\|\widetilde Z^{ij}\|^2_{\mathscr G^{(N)},BMO}\leq (\max_iA^{i,N})^2.
	\end{equation*}
Now fix $(\widetilde z^{ij}, \widehat z^{ii}, \widehat z^{i0})$ and $(\widetilde z^{ij'}, \widehat z^{ii'}, \widehat z^{i0'})$  such that $\sum_{j\neq i}\|\widetilde z^{ij}\|^2_{\mathscr G^{(N)},BMO}+\|\widehat z^{ii}\|^2_{\mathscr G^{(N)},BMO}+\|\widehat z^{i0}\|^2_{\mathscr G^{(N)},BMO}\leq (\max_iA^{i,N})^2$ and $\sum_{j\neq i}\|\widetilde z^{ij'}\|^2_{\mathscr G^{(N)},BMO}+\|\widehat z^{ii'}\|^2_{\mathscr G^{(N)},BMO}+\|\widehat z^{i0'}\|^2_{\mathscr G^{(N)},BMO}\leq (\max_iA^{i,N})^2$ and let $(\widehat Y^i,\widehat Z^{ii},\widehat Z^{i0},\{\widetilde Z^{ij}\}_{j\neq i})$ and $(\widehat Y^{i'},\widehat Z^{ii'},\widehat Z^{i0'},\{\widetilde Z^{ij'}\}_{j\neq i})$ be the corresponding solutions. Then the same argument as above yields that
\begin{align*}
	&~\|\widehat Z^{ii}-\widehat Z^{ii'}\|^2_{\mathscr G^{(N)},BMO}+\|\widehat Z^{i0}-\widehat Z^{i0'}\|^2_{\mathscr G^{(N)},BMO}+\sum_{j\neq i}\|\widetilde Z^{ij}-\widetilde Z^{ij'}\|^2_{\mathscr G^{(N)},BMO}\\
\leq&~20\|b^i\|^2_{\mathscr G^{(N)},BMO}\|\widehat z^{ii}-\widehat z^{ii'}\|^2_{\mathscr G^{(N)},BMO}\\
&~+20\|\alpha^i\|^2\sum_{j\neq i}\left\| \widetilde z^{ij} +\widetilde z^{ij'}+\theta^i\frac{\widehat z^{jj}+\widehat z^{jj'}+2\breve Z^j+\frac{2b^j}{\alpha^j}}{N-1} \right\|^2_{\mathscr G^{(N)},BMO}\sum_{j\neq i}\left\| \widetilde z^{ij}-\widetilde z^{ij'} +\frac{\theta^i(\widehat z^{jj}-\widehat z^{jj'})}{N-1} \right\|^2_{\mathscr G^{(N)},BMO}\\
&~+\frac{20\|\theta^i\|^2}{(N-1)^2}\sum_{j\neq i}\|b^j\|^2_{\mathscr G^{(N)},BMO}\sum_{j\neq i}\|\widehat z^{jj}-\widehat z^{jj'}\|^2_{\mathscr G^{(N)},BMO}+20\|b^{i0}\|^2_{\mathscr G^{(N)},BMO}\|\widehat z^{i0}-\widehat z^{i0'}\|^2_{\mathscr G^{(N)},BMO}\\
&~+\frac{20\|\theta^i\|^2}{\left(1-\sum_{j=1}^N\frac{\|\theta^j\|}{N-1+\|\theta^j\|}\right)^2}\|b^{i0}\|^2_{\mathscr G^{(N)},BMO}\left(\frac{1}{N-1}\sum_{j=1}^N\|\widehat z^{j0}-\widehat z^{j0'}\|_{\mathscr G^{(N)},BMO}\right)^2\\
&~+20\|\theta^i\|^2\overline{\|b^{\cdot0}\|^2}^{(-i),N}\frac{1}{N-1}\sum_{j\neq i}\|\widehat z^{j0}-\widehat z^{j0'}\|^2_{\mathscr G^{(N)},BMO}\\
&~+\frac{20\|\theta^i\|^2}{\left(1-\sum_{j=1}^N\frac{\|\theta^j\|}{N-1+\|\theta^j\|}\right)^2}\left(\overline{\|\theta^\cdot b^{\cdot0}\|}^{(-i),N}\right)^2\left(\frac{1}{N-1}\sum_{j=1}^N\|\widehat z^{j0}-\widehat z^{j0'}\|_{\mathscr G^{(N)},BMO}\right)^2\\
		\leq&~20\|b^i\|^2_{\mathscr G^{(N)},BMO}\|\widehat z^{ii}-\widehat z^{ii'}\|^2_{\mathscr G^{(N)},BMO}\\
		&~+40\|\alpha^i\|^2\left\{ 6(\max_iA^{i,N})^2+\frac{12\left(2(\max_iA^{i,N})^2+4R^2+4\overline{\|\frac{b^\cdot}{\alpha^\cdot}\|^2}^{(-i),N}\right)}{N-1} \right\}\sum_{j\neq i}\left\| \widetilde z^{ij}-\widetilde z^{ij'} \right\|_{\mathscr G^{(N)},BMO}^2\\
		&~+\frac{40\|\alpha^i\|^2\|\theta^i\|^2}{(N-1)^2}\left\{ 6(\max_iA^{i,N})^2+\frac{12\left(2(\max_iA^{i,N})^2+4R^2+4\overline{\|\frac{b^\cdot}{\alpha^\cdot}\|^2}^{(-i),N}\right)}{N-1} \right\}\sum_{j\neq i}\left\|\widehat z^{jj}-\widehat z^{jj'}\right\|_{\mathscr G^{(N)},BMO}^2\\
		&~+\frac{20\|\theta^i\|^2}{N-1}\overline{\|b^\cdot\|^2}^{(-i),N}\sum_{j\neq i}\|\widehat z^{jj}-\widehat z^{jj'}\|^2_{\mathscr G^{(N)},BMO}+20\|b^{i0}\|^2_{\mathscr G^{(N)},BMO}\|\widehat z^{i0}-\widehat z^{i0'}\|^2_{\mathscr G^{(N)},BMO}\\
		&~+\frac{20\|\theta^i\|^2}{\left(1-\sum_{j=1}^N\frac{\|\theta^j\|}{N-1+\|\theta^j\|}\right)^2}\|b^{i0}\|^2_{\mathscr G^{(N)},BMO}\left(\frac{1}{N-1}\sum_{j=1}^N\|\widehat z^{j0}-\widehat z^{j0'}\|_{\mathscr G^{(N)},BMO}\right)^2\\
		&~+20\|\theta^i\|^2\overline{\|b^{\cdot0}\|^2}^{(-i),N}\frac{1}{N-1}\sum_{j\neq i}\|\widehat z^{j0}-\widehat z^{j0'}\|^2_{\mathscr G^{(N)},BMO}\\
		&~+\frac{20\|\theta^i\|^2}{\left(1-\sum_{j=1}^N\frac{\|\theta^j\|}{N-1+\|\theta^j\|}\right)^2}\left(\overline{\|\theta^\cdot b^{\cdot0}\|}^{(-i),N}\right)^2\left(\frac{1}{N-1}\sum_{j=1}^N\|\widehat z^{j0}-\widehat z^{j0'}\|_{\mathscr G^{(N)},BMO}\right)^2.
\end{align*}
By \eqref{ass-multi-BSDE} and \eqref{ass-multi-BSDE2}
we get
\begin{equation*}
\begin{split}
&~\|\widehat Z^{ii}-\widehat Z^{ii'}\|^2_{\mathscr G^{(N)},BMO}+\|\widehat Z^{i0}-\widehat Z^{i0'}\|^2_{\mathscr G^{(N)},BMO}+\sum_{j\neq i}\|\widetilde Z^{ij}-\widetilde Z^{ij'}\|^2_{\mathscr G^{(N)},BMO}\\
\leq&~\frac{1}{20}\|\widehat z^{ii}-\widehat z^{ii'}\|^2_{\mathscr G^{(N)},BMO}+\frac{1}{20}\sum_{j\neq i}\left\| \widetilde z^{ij}-\widetilde z^{ij'} \right\|_{\mathscr G^{(N)},BMO}^2+\frac{1}{20(N-1)^2}\sum_{j\neq i}\left\|\widehat z^{jj}-\widehat z^{jj'}\right\|_{\mathscr G^{(N)},BMO}^2\\
&~+\frac{1}{20(N-1)}\sum_{j\neq i}\|\widehat z^{jj}-\widehat z^{jj'}\|^2_{\mathscr G^{(N)},BMO}+\frac{1}{20}\|\widehat z^{i0}-\widehat z^{i0'}\|^2_{\mathscr G^{(N)},BMO}\\
&~+\frac{1}{20}\frac{N}{(N-1)^2}\sum_{j=1}^N\|\widehat z^{j0}-\widehat z^{j0'}\|_{\mathscr G^{(N)},BMO}^2+\frac{1}{20(N-1)}\sum_{j\neq i}\|\widehat z^{j0}-\widehat z^{j0'}\|^2_{\mathscr G^{(N)},BMO}\\
&~+\frac{N}{20(N-1)^2}\sum_{j=1}^N\|\widehat z^{j0}-\widehat z^{j0'}\|_{\mathscr G^{(N)},BMO}^2.
\end{split}
\end{equation*}	
Taking average over $i$ on both sides and rearranging terms, one has
\begin{equation*}
\begin{split}
&~\frac{1}{N}\sum_{i=1}^N\|\widehat Z^{ii}-\widehat Z^{ii'}\|^2_{\mathscr G^{(N)},BMO}+\frac{1}{N}\sum_{i=1}^N\|\widehat Z^{i0}-\widehat Z^{i0'}\|^2_{\mathscr G^{(N)},BMO}+\frac{1}{N}\sum_{i=1}^N\sum_{j\neq i}\|\widetilde Z^{ij}-\widetilde Z^{ij'}\|^2_{\mathscr G^{(N)},BMO}\\
\leq&~\frac{1}{4}\frac{1}{N}\sum_{i=1}^N\|\widehat z^{ii}-\widehat z^{ii'}\|^2_{\mathscr G^{(N)},BMO}+\frac{1}{20}\frac{1}{N}\sum_{i=1}^N\sum_{j\neq i}\left\| \widetilde z^{ij}-\widetilde z^{ij'} \right\|_{\mathscr G^{(N)},BMO}^2+\frac{1}{2}\frac{1}{N}\sum_{i=1}^N\|\widehat z^{i0}-\widehat z^{i0'}\|^2_{\mathscr G^{(N)},BMO}.
\end{split}
\end{equation*}		
Thus, the contraction property follows.
\end{proof}

\subsection{Verification and Convergence to MFG}

\begin{theorem}\label{thm:verification}
Assume Assumption \ref{ass:N-player}, \eqref{ass-theta-benchmark-BSDE}, \eqref{ass-multi-BSDE} and \eqref{ass-multi-BSDE2} hold. Let the space of admissible strategies for each player be $\mathbb H^2_{\mathscr G^{(N)},BMO}$. Then the vector $(\overline \pi^{1,*},\cdots,\overline \pi^{N,*})$ is the unique Nash equilibrium of the $N$-player game \eqref{wealth-exp-N-i}-\eqref{N-cost-i}, where for each $i=1,\cdots,N$, $\overline \pi^{i,*}$ is given by \eqref{candidate-equilibrium-N}. Moreover, the value function for player $i$ is given by 
\[
	V^i(x^i;x^j,j\neq i)=-e^{-\alpha^i(x^i-Y^i_0)},
\]
where $Y^i$ is the backward component of the solution to \eqref{CN:exp-FBSDE-N}.
\end{theorem}
\begin{proof}
	As a corollary of Theorem \ref{thm:CN-wellposedness-exp-multi-BSDE}, $\overline \pi^{i,*}\in\mathbb H^2_{\mathscr G^{(N)},BMO}$ is admissible. By \cite[Theorem 2.3]{Kazamaki-2006}, $M^{i,\overline \pi^{i,*}}_\cdot=\mathcal E\left(-\alpha^i\int_0^\cdot(\overline \pi^{i,*}_s-\overline Z^i_s)\,d\overline W^i_s+\alpha^i\sum_{j\neq i}\int_0^\cdot Z^{ij}_s\,dW^j_s\right)$ is a martingale, where $\mathcal E(\cdot)$ is the stochastic exponential of $\cdot$. By construction, $A^{i,\overline \pi^{i,*}}_\cdot=
	 e^{-\alpha^i\int_0^\cdot(\overline \pi^{i,*}_s\overline b^i_s-f^i_s)\,ds+\frac{(\alpha^i)^2}{2}\int_0^\cdot|\overline \pi^{i,*}_s-\overline Z^i_s|^2\,ds+\frac{(\alpha^i)^2}{2}\sum_{j\neq i}\int_0^\cdot(Z^{ij}_s)^2\,ds}=0$. The construction of $R^{i,\overline \pi^{i,*}}$ implies it is a martingale. For any $\overline \pi\in\mathbb H^2_{\mathscr G^{(N)},BMO}$, $M^{i,\overline \pi}$ is a martingale and $A^{i,\overline \pi}$ is non-decreasing and thus $R^{i,\overline \pi}$ is a supermartingale. As a result, for all $\overline \pi\in\mathbb H^2_{\mathscr G^{(N)},BMO}$
	 $$\mathbb E[-e^{-\alpha^i(X^{i,\overline p}_T-\theta^i\overline X^{-i,*}_T)}]\leq -e^{-\alpha^i(x^i-Y^i_0)}=\mathbb E[-e^{-\alpha^i(X^{i,\overline p^{i,*}}_T-\theta^i\overline X^{-i,*}_T)}]<\infty.$$
\end{proof}

The next corollary of Theorem \ref{thm:CN-wellposedness-exp-multi-BSDE} and Theorem \ref{thm:verification} shows the unique equilibrium of the $N$-player game converges to the unique equilibrium of the corresponding MFG. To emphasize the dependence on $N$, we denote by $(X^{i,N},Y^{i,N},Z^{i,N},Z^{i0,N},Z^{ij,N},j\neq i)_{i=1,\cdots,N}$ the solution to \eqref{CN:exp-FBSDE-N} and by $(\overline \pi^{1,*,N},\cdots,\overline \pi^{N,*,N})$ the equilibrium of \eqref{wealth-exp-N-i}-\eqref{N-cost-i}.
\begin{corollary}[Convergence to MFG]\label{coro:convergence}
	In addition to Assumption \ref{ass:N-player}, we assume \eqref{ass-theta-benchmark-BSDE}, \eqref{ass-multi-BSDE} and \eqref{ass-multi-BSDE2} hold for all $N$. Moreover, we make the following assumptions on $b^i$ and $(\alpha^i,x^i,\theta^i)$: 
	\begin{equation}\label{ass-corollary-convergence}
		\begin{split}
		&~\{ \alpha^i,x^i,\theta^i\}_i\textrm{  is an i.i.d. sequence and}\\
		&~\textrm{there exist measurable functions } b\textrm{ and }b^0\textrm{ such that }\\
		&~b^i_t=b(t,\alpha^i,x^i,\theta^i,W^i_{\cdot\wedge t},W^0_{\cdot\wedge t})\textrm{ and }b^{i0}_t=b^0(t,\alpha^i,x^i,\theta^i,W^i_{\cdot\wedge t},W^0_{\cdot\wedge t}).	 
		\end{split}
	\end{equation}
Then it holds that
\[
	(X^{i,N},Y^{i,N},Z^{i,N},Z^{i0,N})\rightarrow ( X^{i},Y^i,Z^i,Z^{i0})\textrm{ in }\mathbb S^2\times\mathbb S^2\times \mathbb L^2\times\mathbb L^2 ,
\]
where $( X^i,Y^i,Z^i,Z^{i0})$ is the solution to the following mean field system for player $i$ in MFG \eqref{exp-MFG}
\begin{equation*}
\left\{\begin{split}
X^i_t=&~x^i+\int_0^t\left(\overline Z^i_s+\frac{(\overline b^i_s)^\top}{\alpha^i}\right)(\overline b^i_s\,ds+\,d\overline W^i_s)\\
Y^i_t=&~\theta\mathbb E[X^i_T|\mathcal F^0_T]-\int_t^T\left(\overline Z^i_s\overline b^i_s+\frac{|\overline b^i_s|^2}{2\alpha^i}\right)\,ds-\int_t^T\overline Z^i_s\,d\overline W^i_s.
\end{split}\right.
\end{equation*}
Moreover, the optimal strategy and value function of player $i$ in the $N$-player game converges to those of player $i$ in MFG
\[
	\overline \pi^{i,*,N}\rightarrow\overline \pi^{i,*}\textrm{ in }\mathbb L^2\quad\textrm{  and  }\quad\mathbb E\left[-e^{-\alpha^i(X^{i,N}_T-\theta\overline X^{-i}_T)}\right]\rightarrow \mathbb E\left[-e^{-\alpha^i(X^i_T-\theta\mathbb E[X^i_T])}\right].
\]
\end{corollary}
\begin{proof}

The assumption \eqref{ass-corollary-convergence} implies for all $j\neq i$
\[
	\|\theta^i\|=\|\theta^j\|:=\|\theta\|,\quad \|\alpha^i\|=\|\alpha^j\|:=\|\alpha\|,\quad \|b^i\|_\infty=\|b^j\|_\infty:=\|b\|_\infty,\quad \|b^{i0}\|_\infty=\|b^{j0}\|_\infty:=\|b^0\|_\infty.
\]
As a result,
\[
	A^{i,N}=A^{j,N}:=A^N.
\]
Applying It\^o's formula to \eqref{CN:exp-N-multi-BSDE4}, taking expectations on both sides, using H\"older's inequality and Young's inequality, one has
\begin{align}\label{estimate:convergence-1}
&~\mathbb E(\widehat Y^{i,N}_t)^2+\mathbb E \left[\int_t^T(\widehat Z^{ii,N}_s)^2\,ds\right]+\mathbb E \left[\int_t^T(\widehat Z^{i0,N}_s)^2\,ds\right]+\sum_{j\neq i}\mathbb E\left[\int_t^T(\widetilde Z^{ij,N}_s)^2\,ds\right]\nonumber\\
\leq&~2\mathbb E\left[\int_t^T |\widehat Y^{i,N}_s|\left|b^i_s\widehat Z^{ii,N}_s-\frac{\alpha^i}{2}\sum_{j\neq i}\left(\widetilde Z^{ij,N}_s+\theta^i\frac{\widehat Z^{jj,N}_s+\breve Z^{j}_s+\frac{b^j_s}{\alpha^j}}{N-1}\right)^2-\frac{\theta^i}{N-1}\sum_{j\neq i}b^j_s\widehat Z^{jj,N}_s\right| \,ds\right]\nonumber\\
&~+2\mathbb E\left[\int_t^T|\widehat Y^{i,N}_s| \left|\frac{b^{i0}_s\widehat Z^{i0,N}_s}{1+\frac{\theta^i}{N-1}}+\frac{\theta^i b^{i0}_s}{(1+
	\frac{\theta^i}{N-1})}\frac{1}{N-1}\frac{\sum_{j=1}^N\frac{\widehat Z^{j0,N}_s}{1+\frac{\theta^j}{N-1}}}{1-\sum_{j=1}^N\frac{\theta^j}{N-1+\theta^j}}-\frac{\theta^i}{N-1}\sum_{j\neq i}\frac{b^{j0}_s\widehat Z^{j0,N}_s }{1+\frac{\theta^j}{N-1}} \right| \,ds\right]\nonumber\\
&~+2\mathbb E\left[\int_t^T|\widehat Y^{i,N}_s| \left| -\frac{\theta^i}{N-1}\sum_{j\neq i}b^{j0}_s\frac{\frac{\theta^j}{N-1}\frac{\sum_{k=1}^N\frac{\widehat Z^{k0,N}_s}{1+\frac{\theta^k}{N-1}}}{1-\sum_{k=1}^N\frac{\theta^k}{N-1+\theta^k}}}{1+\frac{\theta^j}{N-1}}+B^{i,N}_s \right| \,ds\right]\nonumber\\
\leq&~\frac{1}{10}\mathbb E\left[\int_t^T(\widehat Z^{ii,N}_s)^2\,ds\right]+10\|b\|^2_\infty\mathbb E\left[\int_t^T(\widehat Y^{i,N}_s)^2\,ds\right]+2\|\alpha\|\|\widehat Y^{i,N}\|_\infty\sum_{j\neq i}\mathbb E\left[\int_t^T(\widetilde Z^{ij,N}_s)^2\,ds\right]\nonumber\\
&~+2\|\alpha\|\|\theta\|^2\|\widehat Y^{i,N}\|_\infty\frac{1}{(N-1)^2}\sum_{j\neq i}\mathbb E\left[\int_t^T\left(\widehat Z^{jj,N}_s+\breve Z^{j}_s+\frac{b^j_s}{\alpha^j}\right)^2\,ds\right]\nonumber\\
&~+\frac{1}{10(N-1)}\sum_{j\neq i}\mathbb E\left[\int_t^T(\widehat Z^{jj,N}_s)^2\,ds\right]+10\|\theta\|^2\|b\|^2_\infty\mathbb E\left[\int_t^T(\widehat Y^{i,N}_s)^2\,ds\right]\nonumber\\
&~+\frac{1}{10}\mathbb E\left[\int_t^T(\widehat Z^{i0,N}_s)^2\,ds\right]+10\|b^{0}\|^2_\infty\mathbb E\left[\int_t^T(\widehat Y^{i,N}_s)^2\,ds\right]\nonumber\\
&~+\frac{1}{10}\left(\frac{1}{N-1}+\frac{1}{(N-1)^2}\right)\sum_{j=1}^N\mathbb E\left[\int_t^T(\widehat Z^{j0,N}_s)^2\,ds\right]+\frac{40\|\theta\|^2\|b^{0}\|_\infty^2}{(1-\|\theta\|)^2}\mathbb E\left[\int_t^T(\widehat Y^{i,N}_s)^2\,ds\right]\nonumber\\
&~+\frac{1}{10(N-1)}\sum_{j\neq i}\mathbb E\left[\int_t^T(\widehat Z^{j0,N}_s)^2\,ds\right]+10\|\theta\|^2\|b^{0}\|^2_\infty\mathbb E\left[\int_t^T(\widehat Y^{i,N}_s)^2\,ds\right]\nonumber\\
&~+\frac{1}{10}\left(\frac{1}{N-1}+\frac{1}{(N-1)^2}\right)\sum_{j=1}^N\mathbb E\left[\int_t^T(\widehat Z^{j0,N}_s)^2\,ds\right]+\frac{40\|\theta\|^4\|b^{0}\|^2_\infty}{(1-\|\theta\|)^2}\mathbb E\left[\int_t^T(\widehat Y^{i,N}_s)^2\,ds\right]\nonumber\\
&~+\mathbb E\left[\int_t^T|\widehat Y^{i,N}_t|^2\,dt\right|+\mathbb E\left[\int_0^T(B^{i,N}_s)^2\,ds\right].
\end{align}
Taking average over $i$ from $1$ to $N$ and noting the estimate \eqref{estimate-thm-widehat}, one has
\begin{align*}
		&~\frac{1}{N}\sum_{i=1}^N\mathbb E(\widehat Y^{i,N}_t)^2+\frac{1}{N}\sum_{i=1}^N\mathbb E \left[\int_t^T(\widehat Z^{ii,N}_s)^2\,ds\right]+\frac{1}{N}\sum_{i=1}^N\mathbb E \left[\int_t^T(\widehat Z^{i0,N}_s)^2\,ds\right]+\frac{1}{N}\sum_{i=1}^N\sum_{j\neq i}\mathbb E\left[\int_t^T(\widetilde Z^{ij,N}_s)^2\,ds\right]\nonumber\\
		\leq&~\left(10\|\overline b^{}\|^2_\infty+10\|\theta\|^2\|\overline b\|^2_\infty+\frac{40\|\theta\|^2\|b^{0}\|_\infty^2}{(1-\|\theta\|)^2}+\frac{40\|\theta\|^4\|b^{0}\|^2_\infty}{(1-\|\theta\|)^2}+1\right)\frac{1}{N}\sum_{i=1}^N\mathbb E\left[\int_t^T(\widehat Y^{i,N}_s)^2\,ds\right]\\
		&~+\frac{3}{10}\frac{1}{N}\sum_{i=1}^N\mathbb E\left[\int_t^T(\widehat Z^{ii,N}_s)^2\,ds\right]+\frac{7}{10}\frac{1}{N}\sum_{i=1}^N\mathbb E\left[\int_t^T(\widehat Z^{i0,N}_s)^2\,ds\right]+2\sqrt 2\|\alpha\| A^{N}\frac{1}{N}\sum_{i=1}^N\sum_{j\neq i}\mathbb E\left[\int_t^T(\widetilde Z^{ij,N}_s)^2\,ds\right]\\
		&~+\frac{2\sqrt 2\|\alpha\|\|\theta\|^2A^{N}}{(N-1)^2}\frac{1}{N}\sum_{i=1}^N\sum_{j\neq i}\mathbb E\left[\int_0^T\left(\widehat Z^{jj,N}_s+\breve Z^j_s+\frac{b^j_s}{\alpha^j}\right)^2\,ds\right]+\frac{1}{5(N-1)^2}\sum_{j=1}^N\mathbb E\left[\int_0^T(\widehat Z^{j0,N}_s)^2\,ds\right]\\
		&~+\frac{1}{N}\sum_{i=1}^N\mathbb E\left[\int_0^T(B^{i,N}_s)^2\,ds\right].
\end{align*}
The second inequality in \eqref{ass-multi-BSDE} implies $2\sqrt 2\|\alpha^i\|A^{N}\leq \frac{1}{40}$. Thus, by rearranging terms we have
\begin{align}\label{estimate:convergence-2}
&~\frac{1}{N}\sum_{i=1}^N\mathbb E(\widehat Y^{i,N}_t)^2+\frac{7}{10}\frac{1}{N}\sum_{i=1}^N\mathbb E \left[\int_t^T(\widehat Z^{ii,N}_s)^2\,ds\right]\nonumber\\
&~+\frac{3}{10}\frac{1}{N}\sum_{i=1}^N\mathbb E \left[\int_t^T(\widehat Z^{i0,N}_s)^2\,ds\right]+\frac{39}{40}\frac{1}{N}\sum_{i=1}^N\sum_{j\neq i}\mathbb E\left[\int_t^T(\widetilde Z^{ij,N}_s)^2\,ds\right]\nonumber\\
\leq&~\left(10\|\overline b^{}\|^2_\infty+10\|\theta\|^2\|\overline b\|^2_\infty+\frac{40\|\theta\|^2\|b^{0}\|_\infty^2}{(1-\|\theta\|)^2}+\frac{40\|\theta\|^4\|b^{0}\|^2_\infty}{(1-\|\theta\|)^2}+1\right)\frac{1}{N}\sum_{i=1}^N\mathbb E\left[\int_t^T(\widehat Y^{i,N}_s)^2\,ds\right]\nonumber\\
&~+\frac{2\sqrt 2\|\alpha\|\|\theta\|^2A^{N}}{(N-1)^2}\frac{1}{N}\sum_{i=1}^N\sum_{j\neq i}\mathbb E\left[\int_0^T\left(\widehat Z^{jj,N}_s+\breve Z^j_s+\frac{b^j_s}{\alpha^j}\right)^2\,ds\right]\nonumber\\
&~+\frac{1}{5(N-1)^2}\sum_{j=1}^N\mathbb E\left[\int_0^T(\widehat Z^{j0,N}_s)^2\,ds\right]+\frac{1}{N}\sum_{i=1}^N\mathbb E\left[\int_0^T(B^{i,N}_s)^2\,ds\right].
\end{align}
Gr\"onwall's inequality and Lemma \ref{lem:estimate-B} imply for each $t\in[0,T]$
\begin{equation*}
	\begin{split}
		&~\frac{1}{N}\sum_{i=1}^N\mathbb E(\widehat Y^{i,N}_t)^2\\
		\leq&~\left\{\frac{2\sqrt 2\|\alpha\|A^{N}}{(N-1)^2}\frac{1}{N}\sum_{i=1}^N\sum_{j\neq i}\mathbb E\left[\int_0^T\left(\widehat Z^{jj,N}_s+\breve Z^j_s+\frac{b^j_s}{\alpha^j}\right)^2\,ds\right]+\frac{1}{5(N-1)^2}\sum_{j=1}^N\mathbb E\left[\int_0^T(\widehat Z^{j0,N}_s)^2\,ds\right]\right.\\
		&~\left.+\frac{1}{N}\sum_{i=1}^N\mathbb E\left[\int_0^T(B^{i,N}_s)^2\,ds\right]\right\}\exp\left\{10\|\overline b^{}\|^2_\infty+10\|\theta\|^2\|\overline b\|^2_\infty  \textcolor{white}{\frac{0}{1}} \right.\\
		&~\left.+\frac{40\|\theta\|^2\|b^{0}\|_\infty^2}{(1-\|\theta\|)^2}+\frac{40\|\theta\|^4\|b^{0}\|^2_\infty}{(1-\|\theta\|)^2}+1 \right\}\\
		\rightarrow&~0\qquad\textrm{as }N\rightarrow\infty.
	\end{split}
\end{equation*}
Going back to \eqref{estimate:convergence-2} we get as $N\rightarrow\infty$
\begin{equation}\label{convergence-average}
	\begin{split}
\frac{1}{N}\sum_{i=1}^N\mathbb E \left[\int_0^T(\widehat Z^{ii,N}_s)^2\,ds\right]+\frac{1}{N}\sum_{i=1}^N\mathbb E \left[\int_0^T(\widehat Z^{i0,N}_s)^2\,ds\right]+\frac{1}{N}\sum_{i=1}^N\sum_{j\neq i}\mathbb E\left[\int_0^T(\widetilde Z^{ij,N}_s)^2\,ds\right]\rightarrow 0.
	\end{split}
\end{equation}
Now we turn to the estimate \eqref{estimate:convergence-1}, from which by rearranging terms we obtain
\begin{align*}
&~\mathbb E(\widehat Y^{i,N}_t)^2+\frac{9}{10}\mathbb E \left[\int_t^T(\widehat Z^{ii,N}_s)^2\,ds\right]+\frac{9}{10}\mathbb E \left[\int_t^T(\widehat Z^{i0,N}_s)^2\,ds\right]+\frac{39}{40}\sum_{j\neq i}\mathbb E\left[\int_t^T(\widetilde Z^{ij,N}_s)^2\,ds\right]\\
\leq&~\left(10\|\overline b^{}\|^2_\infty+10\|\theta\|^2\|\overline b\|^2_\infty+\frac{40\|\theta\|^2\|b^{0}\|_\infty^2}{(1-\|\theta\|)^2}+\frac{40\|\theta\|^4\|b^{0}\|^2_\infty}{(1-\|\theta\|)^2}+1\right)\mathbb E\left[\int_t^T(\widehat Y^{i,N}_s)^2\,ds\right]\\
&~+2\|\alpha\|\|\widehat Y^{i,N}\|_\infty\frac{1}{(N-1)^2}\sum_{j\neq i}\mathbb E\left[\int_0^T\left(\widehat Z^{jj,N}_s+\breve Z^{j,N}_s+\frac{b^j_s}{\alpha^j}\right)^2\,ds\right]+\frac{1}{10(N-1)}\sum_{j\neq i}\mathbb E\left[\int_0^T(\widehat Z^{jj,N}_s)^2\,ds\right]\\
&~+\frac{1}{5}\left(\frac{1}{N-1}+\frac{1}{(N-1)^2}\right)\sum_{j=1}^N\mathbb E\left[\int_0^T(\widehat Z^{j0,N}_s)^2\,ds\right]+\frac{1}{10(N-1)}\sum_{j\neq i}\mathbb E\left[\int_0^T(\widehat Z^{j0,N}_s)^2\,ds\right]\\
&~+\mathbb E\left[\int_0^T(B^{i,N}_s)^2\,ds\right],
\end{align*}
from which Gr\"onwall's inequality, Lemma \ref{lem:estimate-B} and \eqref{convergence-average} imply as $N\rightarrow\infty$
\[
	\sup_{t\in[0,T]}\mathbb E(\widehat Y^{i,N}_t)^2+\mathbb E \left[\int_0^T(\widehat Z^{ii,N}_s)^2\,ds\right]+\mathbb E \left[\int_0^T(\widehat Z^{i0,N}_s)^2\,ds\right]+\sum_{j\neq i}\mathbb E\left[\int_0^T(\widetilde Z^{ij,N}_s)^2\,ds\right]\rightarrow 0.
\]
Applying It\^o's formula again and using the above convergence we have
\[
	\lim_{N\rightarrow\infty}\mathbb E\left[\sup_{0\leq t\leq T}(\widehat Y^{i,N}_t)^2\right]=0.
\]
Using the transformation \eqref{transformation-FB-to-tilde} and \eqref{transformation-tilde-to-hat} we get the desired convergence results.
\end{proof}

\section{Mean Field Games}\label{sec:MFG-exp}
In Section \ref{sec:exp-N} we solved the $N$-player game \eqref{wealth-exp-N-i}-\eqref{N-cost-i} and showed its convergence to the MFG \eqref{exp-MFG} under restrictive assumptions. In this section we solve the MFG \eqref{exp-MFG} directly under a mild assumption on $\theta$. Moreover, when the stock managed by each player is independent of each other,  we show in Section \ref{sec:MFG-IA-exp} that the competition factor $\theta$ can be relaxed to be a random variable that is adapted to the largest $\sigma$-algebra. In particular, we consider the mean field FBSDE \eqref{FBSDE-exp-IA}, whose solvability has its own interest. 

Let $\mathscr G$ be a $\sigma$-algebra independent of $(W,W^0)$ and denote $\mathscr F:=(\mathscr F_t)_{0\leq t\leq T}:=\mathbb F\vee\mathscr G$, where we recall $\mathbb F$ is the filtration generated by $(W,W^0)$.
\subsection{The MFG with Two Risky Assets}\label{sec:MFG-exp-CN}
\begin{assumption}[Assumptions for MFGs with Two Assets]\label{ass:MFG}
\begin{itemize}
	\item The competition factor $\theta$, the initial wealth of the representative player $\mathcal X$ and  the risk aversion parameter $\alpha$ are $\mathscr G$ random variables. Moreover, $\theta$ is valued in $[0,1)$, $\mathcal X\in L^2$ and $\alpha$ is positively valued.
	\item The random return rate for the representative player is bounded, i.e., $b\in\mathbb L^\infty_{\mathscr F}$.
\end{itemize}
\end{assumption}

\begin{theorem}\label{thm:MFG-CA}
	Let Assumption \ref{ass:MFG} hold.
	Then there exists a unique $(X,Y,Z,Z^0)\in\mathbb S^2_{\mathscr F}\times\mathbb S^2_{\mathscr F}\times \mathbb H^2_{\mathscr F,BMO}\times\mathbb H^2_{\mathscr F,BMO}$ satisfying \eqref{FBSDE-exponential-common-noise}. Consequently, there exists a unique equilibrium of the MFG \eqref{exp-MFG} if the space of admissible strategies for the representative player is $\mathbb H^2_{\mathscr F,BMO}$.
\end{theorem}
\begin{proof}
The proof is similar to that of Theorem \ref{thm:CN-wellposedness-exp-multi-BSDE} so we sketch it.  Let 
\[
	\widetilde Y=Y-\theta\int_0^t\mathbb E\left[\left.\overline Z_s\overline b_s+\frac{|\overline b_s|^2}{\alpha}\right|\mathcal F^0_s\right]\,ds-\theta\int_0^t\mathbb E\left[\left.Z^0_s+\frac{b^0_s}{\alpha}\right|\mathcal F^0_s\right]\,dW^0_s,
\]
which implies 
\begin{equation*}
	\left\{\begin{split}
		d\widetilde Y_t=&~\left\{b_t\widetilde  Z_t+\frac{|\overline b_t|^2}{2\alpha}+b^{0}_t\left(\widetilde Z^{0}_t+\frac{\theta}{1-\mathbb E[\theta]}\mathbb E[\widetilde Z^{0}_t|\mathcal F^0_t]+\frac{\theta}{1-\mathbb E[\theta]}\mathbb E\left[\left.\frac{b^{0}_t}{\alpha}\right|\mathcal F^0_t\right]\right)\right.\\
		&~\left.-\theta\mathbb E\left[\left.\frac{|\overline b_t|^2}{\alpha}\right|\mathcal F^0_t\right]-\theta\mathbb E[b_t\widetilde Z_t|\mathcal F^0_t]-\theta\mathbb E[b^{0}_t\breve Z^{0}_t|\mathcal F^0_t]-\frac{\theta}{1-\mathbb E[\theta]}\mathbb E[\theta b^{0}_t|\mathcal F^0_t]\mathbb E[\widetilde Z^{0}_t|\mathcal F^0_t]\right.\\
		&~\left.-\frac{\theta}{1-\mathbb E[\theta]}\mathbb E[\theta b^{0}_t|\mathcal F^0_t]\mathbb E\left[\left.\frac{b^{0}_t}{\alpha}\right|\mathcal F^0_t\right]\right\}\,dt+\widetilde Z_t\,dW_t+\widetilde Z^{0}_t\,dW^0_t,\\
		\widetilde Y_T=&~\theta\mathbb E[\mathcal X],
	\end{split}\right.
\end{equation*}
with $\widetilde Z_t=Z_t$ and  $\widetilde Z^0_t=Z^0_t-\theta\mathbb E[Z^0_t+\frac{b^0_t}{\alpha}|\mathcal F^0_t]\Leftrightarrow Z^0_t=\widetilde Z^0_t+\theta\frac{\mathbb E[\widetilde Z^0_t|\mathcal F^0_t]+\mathbb E[\frac{b^0_t}{\alpha}|\mathcal F^0_t]}{1-\mathbb E[\theta]}$. The well-posedness of the above conditional mean field BSDE in $\mathbb S^2_{\mathscr F}\times\mathbb H^2_{\mathscr F,BMO}\times\mathbb H^2_{\mathscr F,BMO}$ is obvious due to the Lipschitz property of the driver. In order to make the candidate strategy admissible, we need to show $(\widetilde Z,\widetilde Z^0)\in\mathbb H^2_{\mathscr F,BMO}\times\mathbb H^2_{\mathscr F,BMO}$. This can be done by the same approach as that in the proof of Lemma \ref{lem:benchmark-BSDE}. However, in MFG the complexity of the problem is significantly reduced and we can establish the well-posedness result under mild assumptions; see Assumption \ref{ass:MFG}. In particular, we do not need additional assumptions on $\overline b$ like Lemma \ref{lem:benchmark-BSDE}. Indeed, let $\delta$ be a constant small enough such that
\begin{equation}\label{ass-MFG-CA-delta}
\begin{split}
10\|b\|^2_{\infty}\delta \leq \frac{1}{8},\quad 20\left(1+ \frac{\|\theta\|^2}{(1-\mathbb E[\theta])^2}\right)\|b^{0}\|^2_{\infty}\delta\leq \frac{1}{8}.
\end{split}
\end{equation}
Then on $[T-\delta,T]$ the same argument leading to \eqref{estimate-thm-widehat} implies 
\begin{equation*}
	\begin{split}
	&~\frac{3}{10}\|\widetilde Y\|^2_{[T-\delta,T],\infty}+\frac{3}{4}\|\widetilde Z\|_{\mathscr F,[T-\delta,T],BMO}^2+\frac{3}{4}\|\widetilde Z^0\|_{\mathscr F,[T-\delta,T],BMO}^2\\
	\leq&~\frac{5}{2}\left\|\frac{\overline b}{\sqrt{\alpha}}\right\|^4_{\mathscr F,[T-\delta,T],BMO}+\frac{40\|\theta\|^2}{(1-\mathbb E[\theta])^2}\|b^0\|^2_{\mathscr F,[T-\delta,T],BMO}\left\|\frac{b^{0}}{\sqrt{\alpha}}\right\|_{\mathscr F,[T-\delta,T],BMO}^2\\
	&~+10\|\theta\|^2\left\|\frac{b}{\sqrt\alpha}\right\|^4_{\mathscr F,[T-\delta,T],BMO}+\|\theta\|^2(\mathbb E[\mathcal X])^2.
	\end{split}
\end{equation*}
In consideration of \eqref{ass-MFG-CA-delta}, repeating the argument finitely many times implies $\|\widetilde Z\|_{\mathscr F,BMO}+\|\widetilde Z^0\|_{\mathscr F,BMO}<\infty$.

With the wellposedness result especially $(\widetilde Z,\widetilde Z^0)$ located in the BMO space, the verification can be done by the proof of \cite[Theorem 7]{HIM-2005}.

\end{proof}
\begin{remark}
In Theorem \ref{thm:MFG-CA}, $\theta$ is only assumed to be away from $1$. The assumption is significantly reduced compared with Section \ref{sec:exp-N}. It provides sufficient reason to consider MFG although the $N$-player game is solvable. 
\end{remark}

\subsection{MFGs with Independent Assets}\label{sec:MFG-IA-exp}
In this part, we consider the MFG \eqref{exp-MFG-IA}, where each player manages one stock whose dynamics is independent of the others. In this game, we allow the competition factor $\theta$ to be measurable w.r.t. the largest $\sigma$-algebra $\mathscr F_T$, where we recall $\mathscr F_T=\mathcal F_T\vee\mathscr G$. Note that there is no common noise in this section and we use the notation $\mathcal F$ to stand for the filtration generated by the individual noise $W$, without any confusion.

 As $\theta$ is measurable w.r.t. $\mathscr F_T$, in the second step of \eqref{exp-MFG} $\theta\mu$ can be viewed as a general liability like that in \cite[Section 2]{HIM-2005}. The same argument as in  \cite[Section 2]{HIM-2005} implies the optimal strategy and the value function of the representative player can be characterized by

\begin{equation}\label{MF-FBSDE}
\left\{\begin{split}
X_t=&~\mathcal X+\int_0^t\left(b_sZ_s+\frac{|b_s|^2}{\alpha}\right)\,ds+\int_0^t\left(Z_s+\frac{b_s}{\alpha}\right)\,dW_s\\
Y_t=&~\theta\mathbb E[X_T]-\int_t^T\left(b_sZ_s+\frac{|b_s|^2}{2\alpha}\right)\,ds-\int_t^TZ_s\,dW_s.
\end{split}\right.
\end{equation}
The well-posedness of \eqref{MF-FBSDE} can be established in short time intervals or under the assumption that the return rate $b$ is small enough. But it is more reasonable to establish the global well-posedness by assuming that the relative performance factor $\theta$ is small enough. This weak interaction assumption is common in the game theory literature; see \cite{FGHP-2018,FH-2018,Herdegen2019,Horst-2005} among others. 

First, we establish the well-posedness result of \eqref{MF-FBSDE} on a short time interval by a contraction argument.
\begin{proposition}\label{prop:short-time}
	Let Assumption \ref{ass:MFG} hold with $\theta$'s measurability w.r.t. $\mathscr G$ replaced by w.r.t. $\mathscr F_T$. For any $\delta$ satisfying
	\begin{equation}\label{ass-delta-in-prop}
		\delta<\frac{1}{	\left(\frac{4}{3}\|\theta\|^2+\frac{16}{3}\|b\|^2_\infty\|\theta\|^2e^{4\|b\|^2_\infty}\right)\|b\|^2_\infty}   \wedge 1\wedge T,
	\end{equation}
	there exists a unique $(X,Y,Z)\in\mathbb S^2_{\mathscr F}( [T-\delta,T])\times\mathbb S^2_{\mathscr F}( [T-\delta,T])\times\mathbb L^2_{\mathscr F}([T-\delta,T])$ satisfying \eqref{MF-FBSDE} on $[T-\delta,T]$.	
\end{proposition}
\begin{proof}
Replacing $\mathbb E[X_T]$ in \eqref{MF-FBSDE} by $\mu\in \mathbb R$, there exists unique $(Y,Z)$ solving the backward dynamics of \eqref{MF-FBSDE} and taking $Z$ into the forward part of \eqref{MF-FBSDE} we obtain an $X$ and thus a $\mathbb E[X_T]$. Thus, we obtain a mapping from $\mu\in\mathbb R$ to $\mathbb E[X_T]\in\mathbb R$. 
	For any $\mu$ and $\mu'\in \mathbb R$, let $(X,Y, Z)$ and $(X',Y',Z')$ be the corresponding solutions to \eqref{MF-FBSDE} with $\mathbb E[X_T]$ replaced by $\mu$ and $\mu'$, respectively. Thus, $(X-X',Y-Y', Z-Z')$ satisfy for each $t\in[T-\delta,T]$
	\begin{equation*}
	\left\{\begin{split}
	X_t-X'_t=&~\int_{T-\delta}^tb_s\left(Z_s- Z'_s\right)\,ds+\int_{T-\delta}^t\left( Z_s- Z'_s\right)\,dW_s\\
	Y_t-Y'_t=&~\theta(\mu-\mu')-\int_t^Tb_s\left( Z_s-Z'_s\right)\,ds-\int_t^T\left(Z_s- Z'_s\right)\,dW_s.
	\end{split}\right.
	\end{equation*}
	Applying It\^o's formula to the backward dynamics we get
	\begin{equation*}
	\begin{split}
	&~\mathbb E[(Y_t-Y'_t)^2]+\mathbb E\left[\int_t^T(Z_s-Z'_s)^2\,ds\right]\\
	\leq&~\|\theta\|^2|\mu-\mu'|^2+2\|b\|_{\infty}\mathbb E\left[\int_t^T|Y_s-Y'_s||Z_s-Z'_s|\,ds\right]\\
	\leq&~\|\theta\|^2|\mu-\mu'|^2+\frac{1}{4}\mathbb E\left[\int_t^T(Z_s- Z'_s)^2\,ds\right]+4\|b\|^2_{\infty}\mathbb E\left[\int_t^T(Y_s-Y'_s)^2\,ds\right],
	\end{split}
	\end{equation*}
	where we recall by $\|b\|_\infty$ we mean the essential bound on $\Omega\times[0,T]$ not just on $\Omega\times[T-\delta,T]$.
	Gr\"onwall's inequality implies 
	\[
	\mathbb E[(Y_t-Y'_t)^2]\leq \|\theta\|^2e^{4\|b\|^2_{\infty}\delta}|\mu-\mu'|^2
	\]
	and hence
	\[
	\mathbb E\left[\int_{T-\delta}^T(Z_s- Z'_s)^2\,ds\right]\leq \left(\frac{4}{3}\|\theta\|^2+\frac{16}{3}\delta\|b\|^2_{\infty}\|\theta\|^2e^{4\|b\|^2_{\infty}\delta}\right)|\mu-\mu'|^2.
	\]
	From the forward dynamics we have
	\begin{equation*}
	\begin{split}
	|\mathbb E[X_T-X'_T]|^2\leq&~ \|b\|^2_{\infty}\delta\mathbb E\left[\int_{T-\delta}^T(Z_s- Z'_s)^2\,ds\right].
	\end{split}
	\end{equation*}
By \eqref{ass-delta-in-prop}, the contraction follows.
\end{proof}

For any initial point $(t,x,\overline x)$, define $u(t,x,\overline x)=Y^{t,x,\overline x}_t$. Uniqueness of \eqref{MF-FBSDE} on $[T-\delta,T]$ implies $Y^{t,x,\overline x}_s=u(s,X^{t,x,\overline x}_s,\mathbb E[X^{t,x,\overline x}_s])$ for $s\in[t,T]$ and $t\in[T-\delta,T]$. Observe that $(Y,Z)$ is well known once $\mathbb E[X]$ is fixed. Thus, $Y$ depends only on $\mathbb E[X]$ and $Y_s=u(s,\mathbb E[X_s])$ for $s\in[T-\delta,T]$ with $Y_T=u(T,\mathbb E[X_T])=\theta\mathbb E[X_T]$. Moreover, note that the coefficients of \eqref{MF-FBSDE} are random, the decoupling field $u$ is random and we drop the dependence on $\omega$ for simplicity. To extend the solution from $[T-\delta,T]$ to the whole interval, following \cite{MWZZ-2015} we need to prove the decoupling field on each interval is uniformly Lipschitz; in \cite{MWZZ-2015} such a decoupling field is called a \textit{regular decoupling field}. To do so, we consider the variational FBSDE. 
	
	For any $\mathcal X$ and $\mathcal X'$ in $L^2$, let $(X,Y,Z)$ and $(X',Y',Z')$ be solutions to \eqref{MF-FBSDE} corresponding to $\mathcal X$ and $\mathcal X'$, respectively. The triple $(\mathbb E[\nabla X],\nabla Y,\nabla Z)$ satisfies the following variational FBSDE
	\begin{equation}\label{variational-FBSDE}
	\left\{\begin{split}
	\mathbb E[\nabla X_t]=&~1+\int_0^t\mathbb E[b_s\nabla Z_s]\,ds\\
	\nabla Y_t=&~\theta\mathbb E[\nabla X_T]-\int_t^Tb_s\nabla Z_s\,ds-\int_t^T\nabla Z_s\,dW_s,
	\end{split}\right.
	\end{equation}
	where 
	\[
		\mathbb E[\nabla X]=\frac{\mathbb E[X]-\mathbb E[X']}{\mathbb E[\mathcal X]-\mathbb E[\mathcal X']},\quad\nabla Y=\frac{Y-Y'}{\mathbb E[\mathcal X]-\mathbb E[\mathcal X']},\quad\nabla Z=\frac{Z-Z'}{\mathbb E[\mathcal X]-\mathbb E[\mathcal X']}.
	\] 
By \eqref{variational-FBSDE}, $\left(\frac{\nabla Y}{\mathbb E[\nabla X]}, \frac{\nabla Z}{\mathbb E[\nabla X]}\right)$ satisfies the following mean field BSDE
	\begin{equation}\label{characteristic-BSDE-exp}
\widehat Y_t=\theta-\int_t^T\left(b_s\widehat Z_s-\widehat Y_s\mathbb E\left[b_s\widehat Z_s\right]\right)\,ds-\int_t^T\widehat Z_s\,dW_s.
\end{equation}
Following the terminology in \cite{MWZZ-2015}, we call \eqref{characteristic-BSDE-exp} the \textit{characterisitc BSDE} of \eqref{MF-FBSDE}. Moreover, for any $(\widehat Y,\widehat Z)$ satisfying \eqref{characteristic-BSDE-exp}, define $\mathbb E[\nabla X]$ as the unique solution to the ODE
\[
\mathbb E[\nabla X_t]=1+\int_0^t\mathbb E[b_s\widehat Z_s]\mathbb E[\nabla X_s]\,ds
\]
and define
\[
\nabla Y=\widehat Y\mathbb E[\nabla X]\quad\textrm{and}\quad \nabla Z=\widehat Z\mathbb E[\nabla X].
\]
Then $(\mathbb E[\nabla X],\nabla Y,\nabla Z)$ must satsify \eqref{variational-FBSDE}. Thus,  the solution to \eqref{characteristic-BSDE-exp} and the solution to \eqref{variational-FBSDE} have a one-to-one correspondence through
\[
\widehat Y=\frac{\nabla Y}{\mathbb E[\nabla X]}\quad\textrm{and}\quad\widehat Z=\frac{\nabla Z}{\mathbb E[\nabla X]}.
\]
The following proposition establishes the well-posedness result of \eqref{characteristic-BSDE-exp} on the whole interval $[0,T]$.

\begin{proposition}\label{prop:wellposedness-characteristic-BSDE}
Let $\theta$ be measurable w.r.t. $\mathscr F_T$. For each $R$ satisfying
\begin{equation}\label{ass-R-characteristic-BSDE}
	R\leq \frac{1}{4\|b\|_{\mathscr F,BMO}},
\end{equation}
choosing $\|\theta\|$ small enough such that
\begin{equation}\label{ass-theta-characteristic-BSDE}
	\left\{\begin{split}
&~8\left(\|b\|_{\mathscr F,BMO}+1\right)^2(1+ 2\|b\|_{\mathscr F,BMO}e^{2\|b\|_{\mathscr F,BMO}R}R)\|\theta\|^2  \leq R^2\\
&~8(\|b\|_{\mathscr F,BMO}+1)^2\|b\|_{\mathscr F,BMO}\|\theta\|e^{\|b\|_{\mathscr F,BMO}R}\leq \frac{1}{2},
	\end{split}\right.
\end{equation}
there exists a unique $(\widehat Y,\widehat Z)\in\mathbb S^\infty_{{\mathscr F}}\times\mathbb H^2_{{\mathscr F},BMO}$ satisfying the characterisitc BSDE \eqref{characteristic-BSDE-exp}. In particular, $\|\widehat Y\|_\infty+\|\widehat Z\|_{\mathscr F,BMO}\leq R$.
\end{proposition}
\begin{proof}

	Define $\left.\frac{d\mathbb Q}{d\mathbb P}\right|_{\mathscr F_t}=e^{-\int_0^tb_s\,dW_s-\frac{1}{2}\int_0^tb^2_s\,ds}$, which implies $W^{\mathbb Q}:=W+\int_0b_s\,ds$ is a $\mathbb Q$ Brownian motion. Under $\mathbb Q$ we consider
	\begin{equation}\label{comparison-3-exp}
	\widehat Y_t=\theta+\int_t^T\widehat Y_s\mathbb E\left[b_s\widehat Z_s\right]\,ds-\int_t^T\widehat Z_s\,dW^{\mathbb Q}_s.
	\end{equation}
	We solve \eqref{comparison-3-exp} by a contraction argument. For any $R>0$ and $z\in\mathbb H^2_{\mathscr F,BMO,R}$ we consider
	\begin{equation}\label{comparison-4-exp}
	\widehat Y_t=\theta+\int_t^T\widehat Y_s\mathbb E\left[b_sz_s\right]\,ds-\int_t^T\widehat Z_s\,dW^{\mathbb Q}_s.
	\end{equation}
	The formula for linear BSDE yields
	\begin{equation}\label{exp-Yhat-1}
	0\leq \widehat Y_t=\mathbb E^{\mathbb Q}\left[\left.\theta e^{\int_t^T\mathbb E[b_sz_s]\,ds}\right|{\mathscr F}_t\right]\leq \|\theta\|e^{\|b\|_{\mathscr F,BMO}R}.
	\end{equation}
	Applying It\^o's formula to \eqref{comparison-4-exp} we have for each $\tau\in\mathcal T_{{\mathscr F}}$
	\begin{equation}
	\begin{split}
	\widehat Y^2_\tau+\mathbb E^{\mathbb Q}\left[\left.\int_\tau^T\widehat Z^2_s\,ds\right|{\mathscr F}_\tau\right]\leq&~ \|\theta\|^2+2\mathbb E^{\mathbb Q}\left[\left.\int_\tau^T\widehat Y^2_s\mathbb E[|b_sz_s|]\,ds\right|{\mathscr F}_\tau\right]\\
	\leq&~\|\theta\|^2+2\|\widehat Y\|^2_{\infty}\mathbb E\left[\int_0^T|b_sz_s|\,ds\right]\\
	\leq&~\|\theta\|^2+2\|\widehat Y\|^2_\infty\|b\|_{\mathscr F,BMO}\|z\|_{\mathscr F,BMO}\\
	\leq&~(1+ 2\|b\|_{\mathscr F,BMO}e^{2\|b\|_{\mathscr F,BMO}R}R)\|\theta\|^2.
	\end{split}
	\end{equation}
	Thus, by the first inequality in \eqref{ass-theta-characteristic-BSDE} and \cite[Lemma A.1]{Herdegen2019} we get $\|\widehat Z\|_{\mathscr F,BMO}\leq R$. So the mapping $z\in\mathbb H^2_{BMO,R}\rightarrow \widehat Z\in\mathbb H^2_{BMO,R}$ is well-defined. Next we show this mapping is a contraction. Fix $z,z'\in\mathbb H^2_{\mathscr F,BMO,R}$ and let $(\widehat Y,\widehat Z)$ and $(\widehat Y',\widehat Z')$ be the corresponding solutions. It\^o's formula implies
	\begin{equation*}
	\begin{split}
	&~(\widehat Y_\tau-\widehat Y'_\tau)^2+\mathbb E^{\mathbb Q}\left[\left.\int_\tau^T(\widehat Z_s-\widehat Z'_s)^2\,ds\right|{\mathscr F}_\tau\right]\\
	\leq&~ 2\|\widehat Y-\widehat Y'\|^2_{\infty}\mathbb E\left[\int_0^T|b_s||z_s|\,ds\right]+2\|\widehat Y'\|_{\infty}\|\widehat Y-\widehat Y'\|_{\infty}\mathbb E\left[\int^T|b_s||z_s-z'_s|\,ds\right]\\
	\leq&~2\|b\|_{\mathscr F,BMO}\|z\|_{\mathscr F,BMO}\|\widehat Y-\widehat Y'\|^2_{\infty}+2\|b\|_{\mathscr F,BMO} \|\widehat Y'\|_{\infty} \|\widehat Y-\widehat Y'\|_{\infty}\|z-z'\|_{\mathscr F,BMO}\\
	\leq&~2\|b\|_{\mathscr F,BMO}R\|\widehat Y-\widehat Y'\|^2_{\infty}+2\|b\|_{\mathscr F,BMO}\|\theta\|e^{\|b\|_{\mathscr F,BMO}R}\|\widehat Y-\widehat Y'\|_{\infty}\|z-z'\|_{\mathscr F,BMO}\\
	\leq&~(2\|b\|_{\mathscr F,BMO}R+\|b\|_{\mathscr F,BMO}\|\theta\|e^{\|b\|_{\mathscr F,BMO}R})\|\widehat Y-\widehat Y'\|^2_{\infty}+\|b\|_{\mathscr F,BMO}\|\theta\|e^{\|b\|_{\mathscr F,BMO}R}\|z-z'\|^2_{\mathscr F,BMO}.
	\end{split}
	\end{equation*}
	Thus, the contraction property follows from \cite[Lemma A.1]{Herdegen2019}, \eqref{ass-R-characteristic-BSDE} and the second inequality of \eqref{ass-theta-characteristic-BSDE}. 
\end{proof}	
With Proposition \ref{prop:wellposedness-characteristic-BSDE}, we can extend the short time solution in Proposition \ref{prop:short-time} to the whole interval.
\begin{theorem}\label{thm:MFG-IA}
Let Assumption \ref{ass:MFG} holds with the measurability of $\theta$ w.r.t. to $\mathscr G$ replaced by w.r.t. $\mathscr F_T$. For a fixed $R$ satisfying \eqref{ass-R-characteristic-BSDE}, let $\theta$ satisfy \eqref{ass-theta-characteristic-BSDE}. 
Then there exists a unique $(Y,Z)\in\mathbb S^2_{{\mathscr F}}\times\mathbb L^2_{{\mathscr F}}$ satisfying \eqref{MF-FBSDE}. As a result, there exists a unique equilibrium of the MFG \eqref{exp-MFG-IA}.
\end{theorem}
\begin{proof}
	In Proposition \ref{prop:short-time} we have shown there exists a unique solution to \eqref{MF-FBSDE}
 on $[T-\delta,T]$ with $\delta$ satisfying \eqref{ass-delta-in-prop}. Next, 	
we will prove the well-posedness result of the FBSDE \eqref{MF-FBSDE} on $[T-2\delta,T-\delta]$ by choosing $\delta$ further smaller
	\begin{equation}\label{FBSDE_T-2delta_T-delta}
	\left\{\begin{split}
	X_t=&~X_{T-2\delta}+\int_{T-2\delta}^t\left(b_sZ_s+\frac{|b_s|^2}{\alpha}\right)\,ds+\int_{T-2\delta}^t\left(Z_s+\frac{b_s}{\alpha}\right)\,dW_s\\
	Y_t=&~u(T-\delta,\mathbb E[X_{T-\delta}])-\int_t^{T-\delta}\left(b_sZ_s+\frac{|b_s|^2}{2\alpha}\right)\,ds-\int_t^{T-\delta}Z_s\,dW_s,
	\end{split}\right.
	\end{equation}
	where $u$ is the decoupling field constructed from the solution of \eqref{MF-FBSDE} on $[T-\delta,T]$.
By Proposition \ref{prop:wellposedness-characteristic-BSDE} it holds that $$|u(T-\delta,\mathbb E[X_{T-\delta}])-u(T-\delta,\mathbb E[X'_{T-\delta}])|\leq R|\mathbb E[X_{T-\delta}]-\mathbb E[X'_{T-\delta}]|.$$ Replacing $\mathbb E[X_{T-\delta}]$ by $\mu$ the FBSDE \eqref{FBSDE_T-2delta_T-delta} is decoupled and we can get a mapping from $\mu\in\mathbb R$ to $\mathbb E[X_{T-\delta}]\in\mathbb R$, where $X$ is the solution to \eqref{FBSDE_T-2delta_T-delta} with terminal condition replaced by $u(T-\delta,\mu)$. Next we show this is a contraction. It\^o's formula implies
	\begin{equation*}
	\begin{split}
	&~\mathbb E[(Y_t-Y'_t)^2]+\mathbb E\left[\int_t^{T-\delta}( Z_s-Z'_s)^2\,ds\right]\\
	\leq&~R^2|\mu-\mu'|^2+2\|b\|_\infty\mathbb E\left[\int_t^{T-\delta}|Y_s-Y'_s|| Z_s- Z'_s|\,ds\right]\\
	\leq&~R^2|\mu-\mu'|^2+\frac{1}{4}\mathbb E\left[\int_t^{T-\delta}( Z_s- Z'_s)^2\,ds\right]+4\|b\|^2_\infty\mathbb E\left[\int_t^{T-\delta}(Y_s-Y'_s)^2\,ds\right],
	\end{split}
	\end{equation*}
	together with Gr\"onwall's inequality we get
	\[
	\mathbb E\left[\int_{T-2\delta}^{T-\delta}( Z_s- Z'_s)^2\,ds\right]\leq \left(\frac{4}{3}R^2+\frac{16}{3}\delta\|b\|^2_\infty R^2e^{4\|b\|^2_\infty\delta}\right)|\mu-\mu'|^2.
	\]
	The forward dynamics implies
	\begin{equation*}
	\begin{split}
	|\mathbb E[X_{T-\delta}-X'_{T-\delta}]|^2\leq&~ \|b\|^2_\infty\delta\mathbb E\left[\int_{T-2\delta}^{T-\delta}(Z_s- Z'_s)^2\,ds\right].
	\end{split}
	\end{equation*}
	Thus, in addition to \eqref{ass-delta-in-prop} we let $\delta$ be further smaller such that the following inequality holds
	\begin{equation}\label{ass-theta-MFG-IA-2}
	\|b\|^2_\infty\delta \left(\frac{4}{3}R^2+\frac{16}{3}\|b\|^2_\infty R^2e^{4\|b\|^2_\infty}\right)<\frac{1}{2},
	\end{equation}
	the map $\mu\rightarrow \mathbb E[X_{T-\delta}]$ is a contraction.
	
Noting the uniformly Lipschitz property of the decoupling field by Proposition \ref{prop:wellposedness-characteristic-BSDE} and repeating the above analysis by choosing $\delta$ satisfying
	\begin{equation}\label{theta-in-mainthm-MFG}
	\delta<\frac{1}{2\|b\|^2_\infty R^2\left(\frac{4}{3}+\frac{16}{3}\|b\|^2_\infty e^{4\|b\|^2_\infty}\right)}\wedge \frac{1}{2\|b\|^2_\infty \|\theta\|^2\left(\frac{4}{3}+\frac{16}{3}\|b\|^2_\infty e^{4\|b\|^2_\infty}\right)}\wedge 1\wedge T,
	\end{equation}
which implies both \eqref{ass-delta-in-prop} and \eqref{ass-theta-MFG-IA-2}, the FBSDE on any small interval $[T-(j+1)\delta,T-j\delta]$ is solvable.

To show the well-posedness of \eqref{MF-FBSDE}, it is sufficient to choose compatible parameters. 
For each fixed $R$ satisfying \eqref{ass-R-characteristic-BSDE} we choose $\theta$ satisfying  \eqref{ass-theta-characteristic-BSDE},
	which is equivalent to 
	\[
	\frac{1}{\|\theta\|} \geq 16(\|b\|_{\mathscr F,BMO}+1)^2\|b\|_{\mathscr F,BMO}e^{\|b\|_{\mathscr F,BMO}R}\vee \frac{2\sqrt 2(\|b\|_{\mathscr F,BMO}+1)(1+2\|b\|_{\mathscr F,BMO}Re^{2\|b\|_{\mathscr F,BMO}R})^\frac{1}{2}}{R}.
	\]
	So it is sufficient to choose $\delta$ small enough such that
	\begin{equation*}
	\begin{split}
	\delta<&~\frac{1}{2\|b\|^2_\infty R^2\left(\frac{4}{3}+\frac{16}{3}\|b\|^2_\infty e^{4\|b\|^2_\infty}\right)} \wedge 1\wedge T\\
	&~\wedge\left(\frac{32(\|b\|_{\mathscr F,BMO}+1)^4\|b\|^2_{\mathscr F,BMO}e^{2\|b\|_{\mathscr F,BMO}R}}{\|b\|^2_\infty\left(\frac{1}{3}+\frac{4}{3}\|b\|^2_\infty e^{4\|b\|^2_\infty}\right)}\vee\frac{(\|b\|_{\mathscr F,BMO}+1)^2(1+2\|b\|_{\mathscr F,BMO}Re^{2\|b\|_{\mathscr F,BMO}R})}{\|b\|^2_\infty R^2\left(\frac{1}{3}+\frac{4}{3}\|b\|^2_\infty e^{4\|b\|^2_\infty}\right)}\right),
	\end{split}
	\end{equation*}
	which implies \eqref{theta-in-mainthm-MFG}.
	
	It remains to prove the solution of \eqref{MF-FBSDE} characterizes the unique equilibrium of \eqref{exp-MFG-IA}. To do so, it is sufficient to prove $Z\in\mathbb H^2_{\mathscr F,BMO}$. This result is standard since the terminal condition $\theta\mathbb E[X_T]$ is well-known and bounded and the driver is subquadratic; see e.g. \cite{Kobylanski-2000}.
\end{proof}

\section{Explicitly Solvable Examples Beyond Constant Settings}\label{sec:example}
In this section, we consider a special case under our investment framework and derive the equilibria for both the $N$-player game and the MFG in closed forms. Moreover, based on the closed form solutions, the financial interpretation is provided.

\subsection{An N-player game}\label{sec:example-N-game}
For each $i$, assume $(\alpha^i,\theta^i,x^i)$ are $\mathscr G$-random variables, which are independent of all Brownian motions and $\alpha^i>0$, $\Pi_{i=1}^n\theta^i<1$. Let the return processes for the two stocks managed by player $i$ be $b^i(t,\alpha^i,\theta^i,x^i)$ and $b^{i0}(t,\alpha^i,\theta^i,x^i)$, respectively, for two bounded and measurable functions $b^i$ and $b^{i0}$. Recall that the volatilities are assumed to be $1$ for simplicity.

Define the $Z$-component as follows: for each $i=1,\cdots,N$
\begin{equation}\label{solution-Zi-1}
	Z^{ii}=0,\quad Z^{ij}=\frac{\theta^i}{N-1}\frac{b^j}{\alpha^j},\quad j\neq i
\end{equation}
and
\begin{equation}\label{solution-Zi-2}
	Z^{i0}=\frac{\theta^i}{N-1+\theta^i}\frac{\sum_{j=1}^N\frac{\theta^j}{N-1+\theta^j}\sum_{j=1}^N\frac{b^{j0}}{\alpha^j}-\sum_{j=1}^N\frac{\theta^ib^{i0}}{(N-1+\theta^j)\alpha^j}}{1-\sum_{j=1}^N\frac{\theta^j}{N-1+\theta^j}}+\frac{\theta^i}{N-1+\theta^i}\sum_{j\neq i}\frac{b^{j0}}{\alpha^j}.
\end{equation}
Note that the denominator $1-\sum_{j=1}^N\frac{\theta^j}{N-1+\theta^j}>0$ if and only if $\Pi_{j=1}^N\theta^j<1$. 
The resulting strategy for player $i$ is 
\begin{equation}\label{example:strategy-N}
	\overline\pi=(\pi^i,\pi^{i0})=\left(\frac{b^i}{\alpha^i},Z^{i0}+\frac{b^{i0}}{\alpha^i} \right).
\end{equation}
Now we verify that \eqref{solution-Zi-1} and \eqref{solution-Zi-2} induce a solution of \eqref{CN:exp-FBSDE-N}. Taking \eqref{solution-Zi-1} and \eqref{solution-Zi-2} into the forward equation and then the backward equation, we obtain
\begin{equation*}
	\begin{split}
	Y^i_t=&~\theta^i\overline x^{-i}+\frac{\theta^i}{N-1}\sum_{j\neq i}\int_0^T\left(Z^{j0}_sb^{j0}_s+\frac{|\overline b^j_s|^2}{\alpha^j}  \right)\,ds-\int_t^T\left(Z^{i0}_sb^{i0}_s+\frac{|\overline b^i_s|^2}{2\alpha^i}-\frac{\alpha^i}{2}\sum_{j\neq i}(Z^{ij}_s)^2\right)\,ds\\
	&~+\frac{\theta^i}{N-1}\sum_{j\neq i}\int_0^t\frac{b^j_s}{\alpha^j_s}\,dW^j_s+\frac{\theta^i}{N-1}\sum_{j\neq i}\int_0^t\left(Z^{j0}_s+\frac{b^{j0}_s}{\alpha^j}\right)\,dW^0_s\\
	&~+\sum_{j\neq i}\int_t^T\left(\frac{\theta^i}{N-1}\frac{b^j_s}{\alpha^j}-Z^{ij}_s \right)\,dW^j_s+\int_t^T\left( \frac{\theta^i}{N-1}\sum_{j\neq i}\left( Z^{j0}_s+\frac{b^{j0}_s}{\alpha^j}\right)-Z^{i0}_s \right)\,dW^0_s.
	\end{split}
\end{equation*}
With $Z$ defined as in \eqref{solution-Zi-1} and \eqref{solution-Zi-2}, the integrands of the stochastic integral w.r.t. $dW^j$ and $dW^0$ from $t$ to $T$ vanish, that is,
\begin{equation}\label{linear-system-Zi0}
	\frac{\theta^i}{N-1}\frac{b^j}{\alpha^j}-Z^{ij}=0,\qquad   \frac{\theta^i}{N-1}\sum_{j\neq i}\left( Z^{j0}+\frac{b^{j0}}{\alpha^j}\right)-Z^{i0}=0.
\end{equation}
 It results in that $Y^i_t$ is measurable w.r.t. $\mathscr F^{(N)}_t$.

\subsection{MFG: Random Return Processes}\label{sec:MFG-I}
Let $(\alpha,\theta,\mathcal X)$ be $\mathscr G$-random variables with $\mathscr G$ independent of all Browian motions, and the return processes $(b,b^0)$ are assumed to be $$( b_t,b^0_t) = (b(t,\alpha, \theta, \mathcal X), b^0(t,\alpha, \theta, \mathcal X)), $$ for some bounded and measurable functions $(b,b^0)$. We recall the assumption on the volatilities, $ \sigma, \sigma^0 = 1 $. Define the $Z$-component as follows
\begin{equation}\label{solution:Z-random}
(Z, Z^0) = \overline{Z} = \left(0, \frac{\theta}{1 -\mathbb E[ \theta]}\mathbb E\left[\frac{b^0}{\alpha}\right] \right)
\end{equation}
For the above $ \overline{Z} $, we have the investment strategy for the representative agent
\begin{equation}\label{example:strategy-MFG}
\overline\pi = (\pi,\pi^0)=\left(\frac{b}{\alpha}, \frac{\theta}{1 - \mathbb E[\theta]} \mathbb E\left[\frac{b^0}{\alpha}\right] + \frac{b^0}{\alpha} \right).
\end{equation}
We verify directly that \eqref{solution:Z-random} satisfies \eqref{FBSDE-exponential-common-noise}. Indeed, taking \eqref{solution:Z-random} into the forward dynamics of \eqref{FBSDE-exponential-common-noise}, one has
\[
	\theta\mathbb E[X_T|\mathcal F^0_T]=\theta\mathbb E[\mathcal X]+\theta\int_0^T\left(\mathbb E\left[\frac{b^2_s}{\alpha}\right]+\frac{\mathbb E[\theta b^0_s]}{1-\mathbb E[\theta]}\mathbb E\left[\frac{b^0_s}{\alpha}\right]+\mathbb E\left[\frac{(b^0_s)^2}{\alpha}\right]\right)\,ds+\int_0^T\frac{\theta}{1-\mathbb E[\theta]}\mathbb E\left[\frac{b^0_s}{\alpha}\right]\,dW^0_s.
\]
Taking this equality into the backward equation of \eqref{FBSDE-exponential-common-noise}, we have
\begin{equation*}
	\begin{split}
		Y_t=&~\theta\mathbb E[\mathcal X]+\theta\int_0^T\left(\mathbb E\left[\frac{b^2_s}{\alpha}\right]+\frac{\mathbb E[\theta b^0_s]}{1-\mathbb E[\theta]}\mathbb E\left[\frac{b^0_s}{\alpha}\right]+\mathbb E\left[\frac{(b^0_s)^2}{\alpha}\right]\right)\,ds+\int_0^T\frac{\theta}{1-\mathbb E[\theta]}\mathbb E\left[\frac{b^0_s}{\alpha}\right]\,dW^0_s\\
		&~-\int_t^T\left( \frac{\theta}{1-\mathbb E[\theta]}\mathbb E\left[\frac{b^0_s}{\alpha}\right]b^0_s+\frac{|\overline b_s|^2}{2\alpha} \right)\,ds-\int_t^T\frac{\theta}{1-\mathbb E[\theta]}\mathbb E\left[ \frac{b^0_s}{\alpha}\right]\,dW^0_s\\
		=&~\theta\mathbb E[\mathcal X]+\theta\int_0^T\left(\mathbb E\left[\frac{b^2_s}{\alpha}\right]+\frac{\mathbb E[\theta b^0_s]}{1-\mathbb E[\theta]}\mathbb E\left[\frac{b^0_s}{\alpha}\right]+\mathbb E\left[\frac{(b^0_s)^2}{\alpha}\right]\right)\,ds+\int_0^t\frac{\theta}{1-\mathbb E[\theta]}\mathbb E\left[\frac{b^0_s}{\alpha}\right]\,dW^0_s\\
		&~-\int_t^T\left( \frac{\theta}{1-\mathbb E[\theta]}\mathbb E\left[\frac{b^0_s}{\alpha}\right]b^0_s+\frac{|\overline b_s|^2}{2\alpha} \right)\,ds.
	\end{split}
\end{equation*}
Note that the coefficients are $\mathscr G$-measurable. Thus, we obtain an $\mathscr F$-adapted solution. Moreover, the strategy \eqref{example:strategy-MFG} of MFG is a limit of the strategy \eqref{example:strategy-N} of the $N$-player game as $N\rightarrow\infty$.

%

\subsection{Discussion of the solutions}
This section provides some financial interpretation of the equilibria obtained in Section \ref{sec:example-N-game} and Section \ref{sec:MFG-I}. The analytical solutions we obtained for MFG and the $N$-player game share similar structures and shed lights on each other.

{First}, the strategy $\pi^i$ in the $N$-player game (or $\pi$ in MFG) for the individual stock is of Merton type, since the individual stock dynamic in our setting is only driven by the idiosyncratic noise. For the $N$-player game, the hedging demand of player $ i $ for the her individual stock is $ 0 $, apart from the Merton proportion since the return process $b^{i}$ is independent of the noise. This explains why $Z^{ii}=0$ in the $N$-player game. {Second}, the terminal condition of the $N$-player FBSDE \eqref{CN:exp-FBSDE-N} results in the coupled structure of the solution component, which requires $ Z^{i0} $ satisfying a linear system \eqref{linear-system-Zi0}. From \eqref{linear-system-Zi0}, the investment $\pi^{i0}$ can be rewritten as
\[
\pi^{i0}=\frac{\theta^i}{N-1}\sum_{j\neq i}\pi^{j0}+\frac{b^{i0}}{\alpha^i}.
\]

The structure of $\pi^{i0}$ is clear. It consists of a weighted average of competitors' investment strategies and a Merton type portfolio. That is, player $i$'s investment into the stock driven by the common noise $W^0$ is a deviation from the one without competition. 
When $\theta^i$ is large, or player $i$ is more competitive, she is more willing to mimic her competitors' strategies and is encouraged to deviate more from her own Merton portfolio. Otherwise, she will be less aggressive and keep her investment around the Merton one.  
In our MFG, it no longer requires $Z^0$ solving a linear system but a linear equation instead. This linear equation suggests 
\[
	\pi^0=\theta\mathbb E[\pi^0]+\frac{b^0}{\alpha},
\]
where $\mathbb E[\pi^0]$ is considered as the aggregation of the competitors' strategies in the MFG. Thus, $\pi^0$ has a similar interpretation as $\pi^{i0}$ above. Moreover, the expression \eqref{example:strategy-MFG} suggests that
the investment strategy $ \pi^0 $ is a weighted sum of the representative player's own Merton proportion $\frac{b^0}{\alpha}$ and the aggregation of competitors' Merton proportion $\mathbb E\left[\frac{b^0}{\alpha}\right]$. The weight is characterized by the constant $ \frac{\theta}{1 - \mathbb E[\theta]} $ appearing in \eqref{example:strategy-MFG}, which captures that the strategy is not only proportional to the competition itself but also to the interaction effect. When $\theta$ is small, the representative player's own Merton proportion dominates. When $\mathbb E[\theta]$ is large, the aggregation effect dominates. Third, the term $Z^{ij}$ in Section \ref{sec:example-N-game} appears because player $i$'s value function is influenced by the risk $W^j$ of player $j$; it is proportional to player $j$'s investment into the stock driven by $W^j$. In the MFG, competitors' risk does not influence the representative player any more so that $Z^{ij}$ vanishes as $N$ goes to infinity.


Our examples can be compared with \cite{ET-2015} and \cite{LZ-2019}. On one hand, our investment strategies in Section \ref{sec:example-N-game} and Section \ref{sec:MFG-I} are distinguished from \cite{LZ-2019} by the fact that our examples go beyond constant settings while in \cite{LZ-2019}, the unique constant equilibrium is found in a Black-Scholes market where the return processes as well as the the volatilities for the idiosyncratic and common noise are constants. In the present setting, the returns for all stocks are time-dependent. Note that the volatilities are assumed to be $1$ merely for simplicity; they can be time-dependent as well. Moreover, our examples are  distinguished from \cite{ET-2015} by the fact that we allow random return rates as well as each player trades two stocks, not merely a common one as in \cite{ET-2015}. These two differences especially the random return rates make the properties of the investment strategies in \cite{ET-2015} not always true in our case. In particular, the overinvestment property in \cite[Proposition 5.5]{ET-2015} is not necessarily true in our examples. 
On the other hand, our calculation results also suggest some similarities to \cite{ET-2015} and \cite{LZ-2019}. First, if the return rate $b^0$ in our MFG is assumed to be positive as that in \cite{LZ-2019}, then our $\pi^0$ share similar properties with \cite{LZ-2019}; in particular, the representative investor invests more in the stock driven by the common noise if she is more competitive or if her competitors are more competitive. 
Second, in \cite[Proposition 5.5]{ET-2015}, it is shown that the overinvestment in risky assets occurs in more competitive markets, that is, in their setting the absolute value of the investment $|\pi^{i0}|$ is non-decreasing w.r.t. $\theta^j$ for all $j=1,\cdots,N$. This result is obviously true in our MFG when the return rate $b^0$ does not depend on $(\alpha,\theta,\mathcal X)$. Generally, this is not  necessarily true in our $N$-player game and MFG, as mentioned above. Note that when $b^0$ is independent of $(\alpha,\theta,\mathcal X)$, the stock driven by the common noise becomes the common stock for all players.

%
%
%
%
%

\section{Concluding Remarks}
In this paper we study an $N$-player exponential utility game and a mean field exponential utility game in a non-Markovian setting. The well-posedness result of the $N$-player game is established  by solving a novel multi-dimensional FBSDE with quadratic growth and unbounded terminal condition. The convergence from $N$-player game to the MFG is also established. Moreover, when the stock dynamics for each player is independent, we allow a more general competition factor and the equilibrium of MFG is characterized by a novel mean field FBSDE, whose global solvability has its own interest. 

In this paper we only consider exponential utility case. The case of power utility deserves at least an equal treatment. In the power utility case, a more complicated quadratic multi-dimensional FBSDE as well as a novel mean field quadratic FBSDE is obtained. We consider this case in a separate paper. Moreover, in a forthcoming companion paper using the probabilistic analysis we answer an open question in the literature about whether a non-constant equilibrium exists or not in the constant setting like \cite{LZ-2019,LS-2020,Reis2020}. 
Finally, other interesting topics including investment-consumption games (like \cite{LS-2020}) and games with model uncertainty will be studied in our future works.

\begin{appendix}
\section{Appendix}
The following estimate is used in the proof of Lemma \ref{lem:benchmark-BSDE}. 
\begin{lemma}\label{lem:app-estimate-BMO}
Let $f,~g\in\mathbb H^2_{\mathscr G^{(N)},BMO}$.	For each $\tau\in\mathcal T_{\mathscr G^{(N)}}$, it holds that
		\begin{equation*}
	\begin{split}
	\mathbb E\left[\left.\int_\tau^T \mathbb E[f_sg_s|\mathcal F^0_s]\,ds\right|\mathscr G^{(N)}_\tau\right]
	\leq \|f\|_{\mathscr G^{(N)},BMO}\|g\|_{\mathscr G^{(N)},BMO}.
	\end{split}
	\end{equation*}
\end{lemma}
\begin{proof}
	\textbf{Step 1:} By \cite[Problem 2.17(i), Chapter 1]{KS-1991}, it holds that
	\[
	\mathbb E[Z|\mathscr G^{(N)}_\tau]=\mathbb E[Z|\mathscr G^{(N)}_{\tau\wedge t}]\quad\textrm{on }\{\tau\leq t \}
	\]
	as well as
	\[
	\mathbb E[Z|\mathscr G^{(N)}_t]=\mathbb E[Z|\mathscr G^{(N)}_{\tau\wedge t}]\quad\textrm{on }\{t\leq \tau \}.
	\]
	Thus, 
	\[
	\mathbb E[Z|\mathscr G^{(N)}_\tau]1_{\{\tau=t \}}=	\mathbb E[Z|\mathscr G^{(N)}_t]1_{\{\tau=t \}}.
	\]

	\textbf{Step 2: $\tau$ is valued in $\mathbb T\subset[0,T]$, where $\mathbb T$ is a countable set.}
	\begin{align*}
	&~\mathbb E\left[\left.\int_\tau^T \mathbb E[f_sg_s|\mathcal F^0_s]\,ds\right|\mathscr G^{(N)}_\tau\right]=\mathbb E\left[\left.\int_0^T \mathbb E[f_sg_s|\mathcal F^0_s]\,ds\right|\mathscr G^{(N)}_\tau\right]-\int_0^\tau \mathbb E[f_sg_s|\mathcal F^0_s]\,ds\\
	=&~\sum_{t\in\mathbb T}\left\{ \mathbb E\left[\left.\int_0^T \mathbb E[f_sg_s|\mathcal F^0_s]\,ds\right|\mathscr G^{(N)}_\tau\right]-\int_0^\tau \mathbb E[f_sg_s|\mathcal F^0_s]\,ds\ \right\}1_{\{\tau=t\}}\\
	=&~\sum_{t\in\mathbb T}\left\{ \mathbb E\left[\left.\int_0^T \mathbb E[f_sg_s|\mathcal F^0_s]\,ds\right|\mathscr G^{(N)}_t\right]-\int_0^t \mathbb E[f_sg_s|\mathcal F^0_s]\,ds\ \right\}1_{\{\tau=t\}}\quad\textrm{(by Step 1)}\\
	=&~\sum_{t\in\mathbb T}\mathbb E\left[\left.\int_t^T \mathbb E[f_sg_s|\mathcal F^0_s]\,ds\right|\mathscr G^{(N)}_t\right]1_{\{\tau=t\}}\\
	=&~\sum_{t\in\mathbb T}\mathbb E\left[\left.\int_t^T \mathbb E[f_sg_s|\mathcal F^0_s]\,ds\right|\mathcal F^0_t\right]1_{\{\tau=t\}}\quad(\textrm{since }W^0\textrm{ and }\mathscr G,~\{W^i\}_{i=1}^N\textrm{ are independent})\\
	=&~\sum_{t\in\mathbb T}\int_t^T \mathbb E[f_sg_s|\mathcal F^0_t]\,ds1_{\{\tau=t\}}\quad(\textrm{by Fubini theorem and tower property of conditional expectation})\\
	=&~\sum_{t\in\mathbb T} \mathbb E\left[\left.\int_t^Tf_sg_s\,ds\right|\mathcal F^0_t\right]1_{\{\tau=t\}}\quad(\textrm{by Fubini theorem again})\\
	=&~\sum_{t\in\mathbb T} \mathbb E\left\{\left.\mathbb E\left[\left.\int_t^Tf_sg_s\,ds\right|\mathscr G^{(N)}_t\right]\right|\mathcal F^0_t\right\}1_{\{\tau=t\}}\quad(\textrm{by tower property again})\\
	\leq&~\sum_{t\in\mathbb T} \|f\|_{\mathscr G^{(N)},BMO}\|g\|_{\mathscr G^{(N)},BMO}1_{\{\tau=t\}}\\
	=&~\|f\|_{\mathscr G^{(N)},BMO}\|g\|_{\mathscr G^{(N)},BMO}.
	\end{align*}
	
	\textbf{Step 3: $\tau\in\mathcal T_{\mathscr G^{(N)}}$.} Define $\tau_n$ as
	\[
	\tau_n=\sum_{i=1}^n\frac{iT}{n}1_{[\frac{(i-1)T}{n},\frac{iT}{n})}(\tau)+T1_{\{\tau=T\}}.
	\]
	Then $\tau_n$ is an $\mathscr G^{(N)}$ stopping time with at most countable values. Moreover, $\tau_n\geq \tau$ and $\tau_n\rightarrow\tau$. Thus, it holds that
	\begin{equation*}
	\begin{split}
	&~\mathbb E\left[\left.\int_\tau^T \mathbb E[f_sg_s|\mathcal F^0_s]\,ds\right|\mathscr G^{(N)}_\tau\right]\\
	=&~\mathbb E\left\{\left.\mathbb E\left[\left.\int_\tau^T \mathbb E[f_sg_s|\mathcal F^0_s]\,ds\right|\mathscr G^{(N)}_{\tau_n}\right]\right|\mathscr G^{(N)}_\tau\right\}\\
	=&~\mathbb E\left\{\left.\mathbb E\left[\left.\int_{\tau_n}^T \mathbb E[f_sg_s|\mathcal F^0_s]\,ds\right|\mathscr G^{(N)}_{\tau_n}\right]+\mathbb E\left[\left. \int_\tau^{\tau_n}\mathbb E[f_sg_s|\mathcal F^0_s]\,ds\right|\mathscr G^{(N)}_{\tau_n}\right] \right|\mathscr G^{(N)}_\tau\right\}\\
	\leq&~\|f\|_{\mathscr G^{(N)},BMO}\|g\|_{\mathscr G^{(N)},BMO}+\mathbb E\left[\left. \int_\tau^{\tau_n}\mathbb E[f_sg_s|\mathcal F^0_s]\,ds\right|\mathscr G^{(N)}_{\tau}\right]\quad(\textrm{by Step 2})\\
	\rightarrow&~ \|f\|_{\mathscr G^{(N)},BMO}\|g\|_{\mathscr G^{(N)},BMO}.
	\end{split}
	\end{equation*}
\end{proof}
By Lemma \ref{lem:app-estimate-BMO},
the next lemma provides an estimate of $B^{i,N}$ in \eqref{CN:exp-N-multi-BSDE4}.
\begin{lemma}\label{lem:estimate-B}
	Under the assumptions in Lemma \ref{lem:benchmark-BSDE}, it holds that
	\begin{align*}
		&~\left\|\sqrt{|B^{i,N}|}\right\|_{\mathscr G^{(N)},BMO}\\
			\leq&~\|\theta^i\|R\overline{\|b^\cdot\|}^{(-i),N}+\|\theta^i\|\|b^i\|_{\mathscr G^{(N)},BMO}R+\frac{2\|\theta^i\|\|b^{i0}\|_{\mathscr G^{(N)},BMO}R}{1-\sum_{k=1}^N\frac{\|\theta^k\|}{N-1+\|\theta^k\|}}\\
			&~+\frac{\|\theta^i\|}{1-\mathbb E[\theta^i]}\|b^{i0}\|_{\mathscr G^{(N)},BMO}R+\frac{2\|\theta^i\|\|b^{i0}\|_{\mathscr G^{(N)},BMO}}{1-\sum_{k=1}^N\frac{\|\theta^k\|}{N-1+\|\theta^k\|}}\overline{\left\|\frac{ b^{\cdot0}}{\alpha^\cdot}\right\|}^{N}\\
			&~+\frac{\|\theta^i\|}{1-\mathbb E[\theta^i]}\|b^{i0}\|_{\mathscr G^{(N)},BMO}\left\|\frac{b^{i0}}{\alpha^i}\right\|_{\mathscr G^{(N)},BMO}+\|\theta^i\|R\overline{\|b^{\cdot0}\|}^{(-i),N}+\|\theta^i\|\|b^{i0}\|_{\mathscr G^{(N)},BMO}R\\
		&~+\frac{2\|\theta^i\|}{1-\sum_{k=1}^N\frac{\|\theta^k\|}{N-1+\|\theta^k\|}}\overline{\|\theta^\cdot b^{\cdot0}\|}^{(-i),N}R+\frac{\|\theta^i\|^2}{1-\mathbb E[\theta^i]}\|b^{i0}\|_{\mathscr G^{(N)},BMO}R\\
		&~+\frac{2\|\theta^i\|}{1-\sum_{k=1}^N\frac{\|\theta^k\|}{N-1+\|\theta^k\|}}\overline{\|\theta^\cdot b^{\cdot0}\|}^{(-i),N}\overline{\left\|\frac{ b^{\cdot0}}{\alpha^\cdot}\right\|}^{N}+\frac{\|\theta^i\|^2}{1-\mathbb E[\theta^i]}\|b^{i0}\|_{\mathscr G^{(N)},BMO}\left\|\frac{b^{i0}}{\alpha^i}\right\|_{\mathscr G^{(N)},BMO}\\
		&~+\|\theta^i\|\overline{\left\|\frac{\overline b^\cdot}{\sqrt{\alpha^\cdot}}\right\|^2}^{(-i),N}+\|\theta^i\|\left\|\frac{\overline b^i}{\sqrt{\alpha^i}}\right\|^2_{\mathscr G^{(N)},BMO}+\frac{\|\theta^i\|}{N-1}\left\|\frac{b^{i0}}{\sqrt{\alpha^i}}\right\|^2_{\mathscr G^{(N)},BMO}\\
		&~+\frac{\|\theta^i\|}{(N-1)^2}\sum_{j\ne i}\|\theta^j\|\left\|\frac{b^{j0}}{\sqrt{\alpha^j}}\right\|^2_{\mathscr G^{(N)},BMO}+\frac{\|\theta^i\|}{N-1}\|b^{i0}\|_{\mathscr G^{(N)},BMO}R.
	\end{align*}
	Moreover, under the assumptions of Corollary \ref{coro:convergence}, it holds that for any $i$
	\[
		\lim_{N\rightarrow\infty}\mathbb E\left[\int_0^T |B^{i,N}_t|^2\,dt\right]=0.
	\]
\end{lemma}
\begin{proof}
	By Lemma \ref{lem:app-estimate-BMO}, H\"older's inequality and Lemma \ref{lem:benchmark-BSDE}, the estimate for $B^{i,N}$ holds.

Under the assumptions of Corollary \ref{coro:convergence} especially \eqref{ass-corollary-convergence}, Yamada-Watanabe result of mean field (F)BSDE \cite[Lemma 3.2]{FGHP-2018} implies that there exist measurable functions $\phi$ and $\phi_0$ such that the solution to \eqref{CN:exp-N-benchmark-BSDE}	admits the expression 
\[
\breve Z^i_t=\phi(t,\alpha^i,x^i,\theta^i,W^i_{\cdot\wedge t},W^0_{\cdot\wedge t})\quad\textrm{and} \quad 	\breve Z^{i0}_t=\phi_0(t,\alpha^i,x^i,\theta^i,W^i_{\cdot\wedge t},W^0_{\cdot\wedge t}).
\]
As a result, 
\begin{equation}\label{app-convergence-1}
	\lim_{N\rightarrow\infty}\mathbb E\left[\int_0^T\left|-\frac{\theta^i}{N-1}\sum_{j\neq i}b^j_t\breve Z^j_t+\theta^i\mathbb E[b^i_t\breve Z^i_t|\mathcal F^0_t]\right|^2\,dt\right]=0.
\end{equation}
Indeed,	
\begin{align*}
		&~\mathbb E\left[\int_0^T\left|-\frac{\theta^i}{N-1}\sum_{j\neq i}b^j_t\breve Z^j_t+\theta^i\mathbb E[b^i_t\breve Z^i_t|\mathcal F^0_t]\right|^2\,dt\right]\\
		\leq&~\mathbb E\left[\int_0^T\left|\frac{1}{N-1}\sum_{j\neq i}b\phi(t,\alpha^j,x^j,\theta^j,W^j_{\cdot\wedge t},W^0_{\cdot\wedge t})-\mathbb E[b\phi(t,\alpha^i,x^i,\theta^i,W^i_{\cdot\wedge t},W^0_{\cdot\wedge t})|\mathcal F^0_t]\right|^2\,dt\right]\\
		\leq&~\frac{1}{(N-1)^2}\sum_{j=1}^N\mathbb E\left[\int_0^T\left|b\phi(t,\alpha^j,x^j,\theta^j,W^j_{\cdot\wedge t},W^0_{\cdot\wedge t})-\mathbb E[b\phi(t,\alpha^i,x^i,\theta^i,W^i_{\cdot\wedge t},W^0_{\cdot\wedge t})|\mathcal F^0_t]\right|^2\,dt \right]\\
		&~+\frac{1}{(N-1)^2}\sum_{j\neq k}\mathbb E\left[\int_0^T\left|b\phi(t,\alpha^j,x^j,\theta^j,W^j_{\cdot\wedge t},W^0_{\cdot\wedge t})-\mathbb E[b\phi(t,\alpha^i,x^i,\theta^i,W^i_{\cdot\wedge t},W^0_{\cdot\wedge t})|\mathcal F^0_t]\right|\right.\\
		&\left.\textcolor{white}{\int_0^T}\times \left|b\phi(t,\alpha^k,x^k,\theta^k,W^k_{\cdot\wedge t},W^0_{\cdot\wedge t})-\mathbb E[b\phi(t,\alpha^i,x^i,\theta^i,W^i_{\cdot\wedge t},W^0_{\cdot\wedge t})|\mathcal F^0_t]\right|\,dt \right]\\
		=&~\frac{1}{(N-1)^2}\sum_{j=1}^N\mathbb E\left[\int_0^T\left|b\phi(t,\alpha^j,x^j,\theta^j,W^j_{\cdot\wedge t},W^0_{\cdot\wedge t})-\mathbb E[b\phi(t,\alpha^i,x^i,\theta^i,W^i_{\cdot\wedge t},W^0_{\cdot\wedge t})|\mathcal F^0_t]\right|^2\,dt \right]\\
		\rightarrow&~0.
\end{align*}
Now we verify the convergence of the fourth line in the definition of $B^{i,N}$. To do so, we divide it into several parts, that is, 
\begin{align*}
		&~\mathbb E\left[\int_0^T\left(-\frac{\theta^i}{N-1}\sum_{j\neq i}b^{j0}_t\frac{\frac{\theta^j}{N-1}\frac{\sum_{k=1}^N\frac{\breve Z^{k0}_t}{1+\frac{\theta^k}{N-1}}}{1-\sum_{k=1}^N\frac{\theta^k}{N-1+\theta^k}}}{1+\frac{\theta^j}{N-1}} +\frac{\theta^i}{1-\mathbb E[\theta^i]}\mathbb E[\theta^ib^{i0}_t|\mathcal F^0_t]\mathbb E[\breve Z^{i0}_t|\mathcal F^0_t]\right)^2\,dt\right]\\
		\leq&~\mathbb E\left[\int_0^T\left(\frac{1}{N-1}\sum_{j\neq i}b^{j0}_t\frac{\frac{\theta^j}{N-1}\frac{\sum_{k=1}^N\frac{\breve Z^{k0}_t}{1+\frac{\theta^k}{N-1}}}{1-\sum_{k=1}^N\frac{\theta^k}{N-1+\theta^k}}}{1+\frac{\theta^j}{N-1}}-\frac{1}{N-1}\sum_{j\neq i}\theta^jb^{j0}_t\frac{\frac{1}{N-1}\sum_{k=1}^N\frac{\breve Z^{k0}_t}{1+\frac{\theta^k}{N-1}}}{1-\sum_{k=1}^N\frac{\theta^k}{N-1+\theta^k}}\right)^2\,dt\right]\\
		&~+\mathbb E\left[\int_0^T\left(\frac{1}{N-1}\sum_{j\neq i}\theta^jb^{j0}_t\frac{\frac{1}{N-1}\sum_{k=1}^N\frac{\breve Z^{k0}_t}{1+\frac{\theta^k}{N-1}}}{1-\sum_{k=1}^N\frac{\theta^k}{N-1+\theta^k}}-\frac{1}{N-1}\sum_{j\neq i}\theta^jb^{j0}_t\frac{\frac{1}{N-1}\sum_{k=1}^N\breve Z^{k0}_t}{1-\sum_{k=1}^N\frac{\theta^k}{N-1+\theta^k}}\right)^2\,dt\right]\\
		&~+\mathbb E\left[\int_0^T\left(\frac{1}{N-1}\sum_{j\neq i}\theta^jb^{j0}_t\frac{\frac{1}{N-1}\sum_{k=1}^N\breve Z^{k0}_t}{1-\sum_{k=1}^N\frac{\theta^k}{N-1+\theta^k}}-\frac{1}{1-\mathbb E[\theta^i]}\frac{1}{N-1}\sum_{j\neq i}\theta^jb^{j0}_t\frac{1}{N-1}\sum_{k=1}^N\breve Z^{k0}_t\right)^2\,dt\right]\\
		&~+\mathbb E\left[\int_0^T\left(\frac{1}{1-\mathbb E[\theta^i]}\frac{1}{N-1}\sum_{j\neq i}\theta^jb^{j0}_t\frac{1}{N-1}\sum_{k=1}^N\breve Z^{k0}_t-\frac{1}{1-\mathbb E[\theta^i]}\mathbb E[\theta^ib^{i0}_t|\mathcal F^0_t]\mathbb E[\breve Z^{i0}_t|\mathcal F^0_t]\right)^2\,dt\right]\\
		:=&~\mathscr {I}^N_1+\mathscr{I}^N_2+\mathscr{I}^N_3+\mathscr{I}^N_4.
\end{align*}
For $\mathscr I^N_1$, it holds that
\[
	\mathscr I^N_1\leq \frac{1}{(N-1)^2}\mathbb E\left[\int_0^T\left( \frac{1}{N-1}\sum_{j\neq i}\theta^jb^{j0}_t\frac{\frac{1}{N-1}\sum_{k=1}^N\frac{\breve Z^{k0}_t}{1+\frac{\theta^k}{N-1}}}{1-\sum_{k=1}^N\frac{\theta^k}{N-1+\theta^k}}\right)^2\,dt\right]\rightarrow 0.
\]
For $\mathscr I^N_2$, it holds that
\[
	\mathscr I^N_2\leq \frac{1}{(N-1)^2}\mathbb E\left[\int_0^T\left(\frac{1}{N-1}\sum_{j\neq i}\theta^jb^{j0}_t\frac{\frac{1}{N-1}\sum_{k=1}^N\breve Z^{k0}_t}{1-\sum_{k=1}^N\frac{\theta^k}{N-1+\theta^k}}\right)^2\,dt\right].
\]
For $\mathscr I^N_4$, it holds that
\begin{align*}
		\mathscr I^N_4\leq&~ \mathbb E\left[\int_0^T\left(\frac{1}{1-\mathbb E[\theta^i]}\frac{1}{N-1}\sum_{j\neq i}\theta^jb^{j0}_t\left(\frac{1}{N-1}\sum_{k=1}^N\breve Z^{k0}_t-\mathbb E[\breve Z^{i0}_t|\mathcal F^0_t]\right)\right)^2\,dt\right]\\
		&~+ \mathbb E\left[\int_0^T\left(\frac{1}{1-\mathbb E[\theta^i]}\left(\frac{1}{N-1}\sum_{j\neq i}\theta^jb^{j0}_t-\mathbb E[\theta^ib^{i0}_t|\mathcal F^0_t]\right)\mathbb E[\breve Z^{i0}_t|\mathcal F^0_t]\right)^2\,dt\right]\\
		\rightarrow&~ 0,
\end{align*}
where the convergence of the first part is due to the uniform boundedness of $\theta^jb^{j0}$ and the same argument leading to \eqref{app-convergence-1}, and the convergence of the second part is due to the uniform boundedness of $\theta^jb^{j0}$, the dominated convergence and law of large numbers.

For $\mathscr I^N_3$, it holds that
\begin{align*}
\mathscr I^N_3\leq&~
\mathbb E\left[\int_0^T\left(\frac{\frac{1}{N-1}\sum_{k=1}^N\frac{\theta^k}{1+\frac{\theta^k}{N-1}}-\frac{1}{N-1}\sum_{k=1}^N\theta^k}{(1-\mathbb E[\theta^i])(1-\sum_{k=1}^N\frac{\theta^k}{N-1+\theta^k})}\frac{1}{N-1}\sum_{j\neq i}\theta^jb^{j0}_t\frac{1}{N-1}\sum_{k=1}^N\breve Z^{k0}_t\right)^2\,dt\right]\\
&~+	\mathbb E\left[\int_0^T\left(\frac{\frac{1}{N-1}\sum_{k=1}^N\theta^k-\mathbb E[\theta^i]}{(1-\mathbb E[\theta^i])(1-\sum_{k=1}^N\frac{\theta^k}{N-1+\theta^k})}\frac{1}{N-1}\sum_{j\neq i}\theta^jb^{j0}_t\left(\frac{1}{N-1}\sum_{k=1}^N\breve Z^{k0}_t-\mathbb E[\breve Z^{i0}_t|\mathcal F^0_t]\right)\right)^2\,dt\right]\\
&~+	\mathbb E\left[\int_0^T\left(\frac{\frac{1}{N-1}\sum_{k=1}^N\theta^k-\mathbb E[\theta^i]}{(1-\mathbb E[\theta^i])(1-\sum_{k=1}^N\frac{\theta^k}{N-1+\theta^k})}\frac{1}{N-1}\sum_{j\neq i}\theta^jb^{j0}_t\mathbb E[\breve Z^{i0}_t|\mathcal F^0_t]\right)^2\,dt\right]\\
\rightarrow&~0,
\end{align*}
where the convergence of the first part is due to the same reason as $\mathscr I^2_N\rightarrow 0$, the convergence of the second part is due to the uniform boundedness of $\theta^i$ and $b^{i0}$ as well as the same argument leading to \eqref{app-convergence-1}, and the convergence of the third part is due to the same reason as $\mathscr I^N_4\rightarrow 0$.

Similarly, other terms in $B^{i,N}$ also converges to $0$ in $\mathbb L^2$. Thus, $B^{i,N}$ converges to $0$ in $\mathbb L^2$.
\end{proof}

The following proposition is frequently used in the main text and we believe it is a standard result. However, it is difficult to locate the proof so we provide one for completeness. 
	\begin{proposition}\label{prop:appendix}
		Let $\sqrt{ |f|}\in \mathbb H^2_{BMO}$ and $\xi\in \mathbb L^\infty$. Then the following BSDE has a unique solution in $\mathbb S^\infty\times\mathbb H^2_{BMO}$
		\[
			dY_t=f_t\,dt+Z_t\,dW_t,\quad Y_T=\xi.
		\]
	\end{proposition}
\begin{proof}
	Consider the truncated BSDE
	\[
		dY_t=f_t\wedge n\,dt+Z_t\,dW_t,\quad Y_T=\xi.
	\]
	By the boundedness of $f\wedge n$, there exists a unique solution $(Y^n,Z^n)\in \mathbb S^2\times \mathbb L^2$. 
	Moreover, $Y^n_t=\mathbb E\left[\left.\xi-\int_t^Tf_s\wedge n\,ds\right|\mathcal F_t\right]$, which implies
		\[
			|Y^n_t|\leq \|\xi\|+\left\|\sqrt{|f|}\right\|_{BMO}^2<\infty.
		\]
	For each $n$ and $m$, consider
	\[
		d(Y^n_t-Y^m_t)=\left(f_t\wedge n-f_t\wedge m\right)\,dt+(Z^n_t-Z^m_t)\,dW_t,\quad Y^n_T-Y^m_T=0.
	\]
It\^o's formula implies
\[
	(Y^n_t-Y^m_t)^2+\int_t^T(Z^n_s-Z^m_s)^2\,ds\leq 2\int_t^T|Y^n_s-Y^m_s||f_s\wedge n-f_s\wedge m|\,ds+2\int_t^T(Y^n_s-Y^m_s)(Z^n_s-Z^m_s)\,dW_s,
\]
where $\int_t^T(Y^n_s-Y^m_s)(Z^n_s-Z^m_s)\,dW_s$ is a true martingale. 
	 Taking expectations on both sides, we have
	 \begin{equation*}
	 	\begin{split}
	 	\mathbb E[(Y^n_t-Y^m_t)^2]+\mathbb E\left[\int_t^T(Z^n_s-Z^m_s)^2\,ds\right]\leq&~ 2\mathbb E\left[\int_0^T|Y^n_s-Y^m_s||f_s\wedge n-f_s\wedge m|\,ds\right]\\
	 	\leq&~ 2\|Y^n-Y^m\|_\infty\mathbb E\left[\int_0^T|f_s\wedge n-f_s\wedge m|\,ds\right]\\
	 	\leq&~4\left(\|\xi\|+\left\|\sqrt{|f|}\right\|^2_{BMO}\right)\mathbb E\left[\int_0^T|f_s\wedge n-f_s\wedge m|\,ds\right]\\
	 	\rightarrow&~0,\quad\textrm{as }m,n\rightarrow\infty,
		\end{split}
	 \end{equation*}
where the convergence in the last line is by the estimate $|f\wedge n-f\wedge m|\leq 2|f|$, which is integrable, and by the dominated convergence.
Moreover, by taking expectations and BDG inequality, it holds that
\[
	\mathbb E\left[\sup_{0\leq t\leq T}|Y^n_t-Y^m_t|\right]\leq \mathbb E\left[\int_0^T|f_s\wedge n-f_s\wedge m|\,ds\right]+\mathbb E\left[\sup_{0\leq t\leq T}\left|\int_t^T(Z^n_s-Z^m_s)\,dW_s\right|\right]\rightarrow 0.
\]
As a result, $\{(Y^n,Z^n)\}_n$ is a Cauchy sequence in $\mathbb S^1\times\mathbb L^2$. Let $(Y,Z)\in \mathbb S^1\times\mathbb L^2$ be the limit and both $Y$ and $Z$ are adapted. Define $Y^o$ by
\[
Y^o_t=\xi-\int_t^Tf_s\,ds-\int_t^TZ_s\,dW_s.
\]
It remains to prove $Y^o$ is adapted. The same argument as above yields

\[
	\mathbb E\left[\sup_{0\leq t\leq T}|Y^n_t-Y^o_t|\right]\leq \mathbb E\left[\int_0^T|f_s-f_s\wedge n|\,ds\right]+\mathbb E\left[\sup_{0\leq t\leq T}\left|\int_t^T(Z^n_s-Z_s)\,dW_s\right|\right]\rightarrow 0.
\]
As a result, $Y^o=Y$, which is adapted. The dynamics implies $Y_t=\mathbb E\left[\left.\xi-\int_t^Tf_s\,ds\right|\mathcal F_t\right]\in\mathbb S^\infty$. It\^o's formula yields for each $\tau\in\mathcal T$
\[
	Y^2_\tau+\int_\tau^T|Z_s|^2\,ds\leq \xi^2+2\int_\tau^T |Y_sf_s|\,ds-2\int_\tau^TY_sZ_s\,dW_s.
\]
Since $Y\in \mathbb S^\infty$ and $Z\in\mathbb L^2$, the above stochastic integral is a true martingale. By taking conditional expectations and supremum on both sides, we get
\[
	\|Z\|^2_{BMO}\leq \|\xi\|^2+ 2\|Y\|_\infty\left\|\sqrt{|f|}\right\|^2_{BMO}<\infty.
\]
The uniqueness result is obvious.
\end{proof}
\end{appendix}

\bibliography{Fu}

\end{document}